\documentclass{article}
\usepackage[utf8]{inputenc}
\usepackage[T1]{fontenc}
\usepackage{amsmath}
\usepackage{amsfonts}
\usepackage{hyperref}
\hypersetup{
    colorlinks=true,
    linkcolor=cyan,
    filecolor=magenta,      
    urlcolor=cyan,
    pdfpagemode=FullScreen,
    }
\usepackage{amssymb}
\usepackage{mathtools}
\usepackage[version=4]{mhchem}
\usepackage{tabularx}
\usepackage[a4paper, total={6.5in, 9in}]{geometry}
\usepackage[most]{tcolorbox}
\usepackage{xcolor}
\usepackage[numbers]{natbib}
\usepackage[noamsbb]{stmaryrd}
\usepackage{bbold}
\usepackage{multirow} 
\usepackage{enumitem}
\usepackage{xcolor}
\usepackage{hyperref}
\usepackage{tikz-cd}
\usepackage{epstopdf}
\usepackage{bm}
\usepackage{pgfplots}
\usepackage{subcaption}
\usepackage{graphicx}
\usepackage[export]{adjustbox}
\usepackage{amsthm}
\usepackage{chngcntr}
\usepackage{stmaryrd}
\usepackage{apptools}
\usepackage{esint}
\usepackage{lmodern}
\usepackage{tikz}
\usetikzlibrary{patterns,positioning,cd,calc}
\pgfplotsset{compat=1.18}
\usepackage{authblk}

\newtheorem{theorem}{Theorem}[section]
\newtheorem{lemma}{Lemma}[section]

\newtheorem{corollary}[theorem]{Corollary}

\newtheorem{remark}{Remark}

\numberwithin{equation}{section}

\def\eps{{\varepsilon}}

\def\div{\mathrm{div\,}}
\def\divn{\mathrm{div}}
\def\curl{\boldsymbol{\mathrm{curl}}\,}
\def\curln{\boldsymbol{\mathrm{curl}}}
\def\curlGamma{\mathrm{curl}_\Gamma\,}
\def\grad{\boldsymbol{\mathrm{grad}}\,}
\def\gradn{\boldsymbol{\mathrm{grad}}}
\def\gradGamma{\boldsymbol{\mathrm{grad}}_\Gamma\,}

\def\Hdiv{{\boldsymbol{H}(\mathrm{div})}}

\def\Hcurl{{\boldsymbol{H}(\boldsymbol{\mathrm{curl}})}}
\def\Hzcurl{\boldsymbol{H}_0(\boldsymbol{\mathrm{curl}})}

\def\ub{{\boldsymbol{u}}}
\def\vb{{\boldsymbol{v}}}
\def\mub{{\boldsymbol{\mu}}}
\def\nb{{\boldsymbol{n}}}

\def\R{{\mathbb{R}}}
\def\fb{\boldsymbol{f}}
\def\taub{\boldsymbol{\tau}}

\def\Hb{\boldsymbol{H}}
\def\Lb{\boldsymbol{L}}
\def\a{\mathsf{a}}

\def\B{\mathsf{B}}
\def\Nedelec{N\'{e}d\'{e}lec }
\def\xb{\boldsymbol{x}}

\def\Wb{\boldsymbol{W}}
\def\supp{\mathrm{supp}}

\def\wb{\boldsymbol{w}}

\def\Poincare{Poincar\'{e} }

\def\Pb{\boldsymbol{P}}
\def\Sigmab{\boldsymbol{\Sigma}}
\def\rb{\boldsymbol{r}}
\def\Vb{\boldsymbol{V}}

\def\Wb{\boldsymbol{W}}

\def\Mb{\boldsymbol{M}}
\def\gb{\boldsymbol{g}}
\def\betab{{\boldsymbol{\beta}}}

\def\L{\mathsf{L}}
\def\Pib{{\boldsymbol{\Pi}}}

\def\INDav{\boldsymbol{\mathcal{I}}^{\mathrm{ND}}_h}

\def\gammat{\gamma_t}
\def\gammax{\gamma_\times}
\def\gamman{\gamma_n}

\def\zerob{\boldsymbol{0}}

\def\Sigmabh{\boldsymbol{\Sigma}_h}

\def\varphib{\boldsymbol{\varphi}}
\def\Fcal{\mathcal{F}}
\def\Wcal{\mathcal{W}}

\def\Tcal{\mathcal{T}}

\def\Rsfb{\boldsymbol{\mathsf{R}}}
\def\etab{\boldsymbol{\eta}}
\def\Asf{\mathsf{A}}
\def\Bsf{\mathsf{B}}
\def\lsf{\mathsf{l}}
\def\Pcal{\mathcal{P}}
\def\Ib{\boldsymbol{I}}
\def\Dsf{\mathsf{D}}

\def\dx{\,\mathrm{d}x}
\def\ds{\,\mathrm{d}s}

\newcommand{\rmk}[1]{\textcolor{red}{\textbf{(#1)}}}

\newcommand{\vertiii}[1]{{\left\vert\kern-0.25ex\left\vert\kern-0.25ex\left\vert #1 
    \right\vert\kern-0.25ex\right\vert\kern-0.25ex\right\vert}}
\newcommand{\jump}[1]{\left[\hspace{-1.5pt}\left[#1\right]\hspace{-1.5pt}\right]}
\newcommand{\avg}[1]{\left\{\hspace{-3pt}\left\{#1\right\}\hspace{-3pt}\right\}}

\title{A Penalty-Free Asymmetric Nitsche's Method for Edge Elements}
\author[]{Tianwei Yu\thanks{Seminar for Applied Mathematics, ETH Z\"{u}rich, \texttt{tianwei.yu@sam.math.ethz.ch}} }
\affil[]{}

\date{}

\begin{document}
\maketitle
\begin{abstract}
    We show the stability of a penalty-free asymmetric Nitsche's method using Nédélec edge elements for solving $\curln$-$\curln$-type problems with tangential Dirichlet boundary conditions imposed weakly. The main result is an inf-sup stability estimate for the asymmetric bilinear form under an isolated patch condition on the tetrahedral mesh. Applications to a $\curln$-elliptic problem and a magnetic advection-diffusion problem are discussed.
\end{abstract}
\section{Introduction}
Let $\Omega\subset\R^3$ be a bounded Lipschitz polyhedral domain with its boundary $\Gamma := \partial\Omega$. For simplicity, we assume throughout that $\Omega$ has trivial topology in the sense that it is simply connected and its boundary is also simply connected. 
In this work, we investigate a \emph{penalty-free asymmetric Nitsche's method} for $\Hcurl$-conforming edge elements. Specifically, we consider the following bilinear form to discretize the $\curln$-$\curln$ operator:
\begin{equation}\label{eq:blasym-intro}
\begin{aligned}\a^\mathrm{asym}_{\curln,h}(\ub_h,\vb_h) &:= \int_\Omega \curl\ub_h \cdot \curl\vb_h \dx - \int_\Gamma  \gammax \curl \ub_h \cdot \gammat \vb_h \ds + \int_\Gamma \gammax \curl \vb_h \cdot \gammat \ub_h \ds
\end{aligned}
\end{equation}
where $\gammat \ub := \nb \times (\ub \times \nb),\,\gammax \ub := \ub \times \nb$ denote the tangential trace and the rotated tangential trace of $\ub$, respectively.
The bilinear form~\eqref{eq:blasym-intro} discretizes the $\curln$-$\curln$ operator and imposes tangential boundary conditions (BCs) \emph{weakly}. Relevant model problems include:

\begin{itemize}
    \item the stationary Maxwell's equation with Coulomb gauge (see, e.g.,~\cite{hiptmair_2002}):
    \begin{equation}\label{eq:cc-intro}
        \begin{aligned}
        \curl\curl\ub  &= \fb && \text{in }\Omega,\\
        \div \ub &= 0 && \text{in }\Omega,\\
        \ub\times\nb &= \zerob && \text{on } \Gamma,
        \end{aligned}
    \end{equation}
    where $\fb$ is assumed to be solenoidal;
    \item the magnetic advection-diffusion equation arising in magnetohydrodynamics (see, e.g.,~\cite{spruit_2016}):
    \begin{equation}\label{eq:mad-intro}
        \begin{aligned}
        \eps \curl\curl\ub + \curl\ub \times \betab + \grad (\ub\cdot\betab) + \alpha \ub &= \fb && \text{in }\Omega,\\
        \ub\times\nb &= \zerob && \text{on }\Gamma\setminus\Gamma^- \\
        \ub &= \zerob && \text{on }\Gamma^-,
        \end{aligned}
    \end{equation}
    where $\eps$ is the diffusion coefficient; $\betab$ is the advection field; $\alpha$ is the reaction coefficient; and $\Gamma^- := \{x\in\Gamma: \betab(x)\cdot\nb(x) < 0\}$ is the inflow boundary.
\end{itemize}
\subsection{Nitsche's method}
Recall the integration-by-parts relation
\begin{equation}
\begin{aligned}
\int_{\Omega} \curl\curl\ub \cdot \vb \dx &= \int_\Omega \curl\ub \cdot \curl\vb \dx - \int_\Gamma \gammax \curl \ub \cdot \gammat \vb \ds
\end{aligned}
\end{equation}
for $\ub, \vb$ smooth.
Let $\Sigmab_h \subset \Hcurl$ be the \Nedelec edge element space and $P_h \subset H^1$ be the Lagrange nodal element space (see definition in~\eqref{def:spaces}).
To impose the tangential Dirichlet BCs $\gammat \ub = \zerob$, standard $\Hcurl$-conforming finite element discretization seeks discrete solutions in $\Sigmab_{h,0} := \Sigmab_h \cap \Hzcurl$ and uses the bilinear form:
\begin{equation}\label{eq:blstrong}
    \a_{\curln}(\ub,\vb) := \int_\Omega\curl\ub \cdot \curl\vb \dx
\end{equation}
when discretizing $\curl \curl \ub$.
As an alternative approach, one can seek discrete solutions directly in $\Sigmab_h$ while the tangential Dirichlet BCs are enforced weakly through \emph{penalization}. Consider the following bilinear form:
\begin{multline}\label{eq:blsym}
    \a^\mathrm{sym}_{\curln,h}(\ub_h,\vb_h) := \int_\Omega \curl\ub_h \cdot \curl\vb_h \dx - \int_\Gamma \gammax \curl \ub_h \cdot \gammat\vb_h \ds - \int_\Gamma \gammax\curl \vb_h \cdot \gammat \ub_h \ds \\
    + \frac{C_{p}}{h}\int_\Gamma \gammat\ub_h \cdot \gammat\vb_h \ds
\end{multline}
where $C_p > 0$ is a user-defined penalty parameter. With $C_p$ chosen sufficiently large, the bilinear form $\a^\mathrm{sym}_{\curln,h}$ can be shown coercive with respect to the $h$-dependent (semi)norm $\|\curl \vb\|_{\Lb^2(\Omega)} + h^{-1/2}\|\gammat \vb\|_{\Lb^2(\Gamma)}$~\cite[][Eq.~3.10]{wouter_2026}.

Compared to~\eqref{eq:blsym}, the target bilinear form~\eqref{eq:blasym-intro} discards the penalty term by adopting an \emph{asymmetric} formulation. As an immediate consequence, it is coercive to a weaker norm $\|\curl \vb\|_{\Lb^2(\Omega)}$, which is not sufficient to guarantee the stability of the method and makes the analysis more challenging.

\subsection{Literature review}
Approaches involving bilinear forms of type~\eqref{eq:blsym} (see also~\eqref{eq:blsymscalar}) are called \emph{Nitsche's method}. The original idea traces back to~\cite{nitsche_1971} for enforcing Dirichlet BCs weakly in the context of scalar elliptic problems (see Section~\ref{sec:scalar} for a brief recap). This approach has been extensively studied and plays a fundamental role in cutFEM/unfitted FEM for interface problems and complex geometries~\cite{hansbo_2002,heinrich_2003,annavarapu_2012,hansbo_2014,buffa_2020,burman_2025} and discontinuous Galerkin (DG) methods~\cite{arnold_2002}. Applications have been extensively investigated for elastic contact problems~\cite{chouly_2015,gustafsson_2020} and fluid-structure interaction~\cite{burman_2014,massing_2015,alauzet_2016}. Furthermore, weak imposition of Dirichlet BCs has shown advantages in treating singularly perturbed problems displaying boundary layers, such as advection-diffusion problems~\cite{schieweck_2008,burman_2006b}, drift-diffusion problems~\cite{yu_2025}, Brinkman-Darcy-Stokes flows~\cite{hansbo_2009,juntunen_2010}, and incompressible Navier-Stokes equations~\cite{bazilevs_2005,bazilevs_2007}. 

Extension of $H^1$-conforming nodal-element-based Nitsche's method to $\Hcurl$-/$\Hdiv$-conforming elements has been a topic of extensive interest. See~\cite{casagrande_2016,roppert_2020,liu_2020} for $\Hcurl$-/$\Hdiv$-interface problems and~\cite{wang_2007,konno_2011,wouter_2026} for $\Hcurl$-/$\Hdiv$-based Stokes flows.

The standard Nitsche's method involves a \emph{symmetric} bilinear form augmented with a boundary $L^2$-penalty term; see~\eqref{eq:blsym} and~\eqref{eq:blsymscalar}. The coefficient of such penalty term must be chosen sufficiently large to guarantee the stability of the method. A variant of Nitsche's method uses an \emph{asymmetric} bilinear form and dispenses with the penalty term (\emph{penalty-free}); see~\eqref{eq:blasym} and~\eqref{eq:blasymscalar}. This approach was first proposed and analyzed in the context of DG methods for advection-diffusion problems~\cite{baumann_1999} (often referred to as the Bauman-Oden method). Its stability on the $H^1$-conforming element was established in~\cite{burman_2012} under suitable mesh conditions. The theory was later extended to linear elasticity problems~\cite{boiveau_2016} and Brinkman-Darcy-Stokes flows~\cite{blank_2018}. Numerical evidence suggests certain benefits of this penalty-free asymmetric variant~\cite{chouly_2015,blank_2018}.

Overall, Nitsche's method provides a consistent, stable, and flexible way of enforcing boundary conditions/interface conditions that stem from certain diffusion terms. An abstract framework (of the symmetric variant) is provided in~\cite{benzaken_2022} for general variational problems.  

\subsection{Outline}

The main result of this work is an inf-sup stability estimate of the bilinear form~\eqref{eq:blasym} (see Lemma~\ref{thm:infsupasym}). To the best of the author's knowledge, such a result does not exist in the literature. It can be viewed as a nontrivial extension of the $H^1$-based result in~\cite{burman_2012} to the $\Hcurl$-case.

The rest of the paper is organized as follows. Basic setup and preliminaries are introduced in Section~\ref{sec:prelim}. In the beginning of Section~\ref{sec:infsup}, a quick review of the Nitsche's method for scalar elliptic boundary value problems is provided. Then, we establish the inf-sup stability estimate for the bilinear form $\a^\mathrm{asym}_{\curln,h}$.  In Sections~\ref{sec:curlcurl} and~\ref{sec:mad}, we discuss the application of the method to the $\Hcurl$-elliptic problem~\eqref{eq:cc-intro} and the magnetic advection-diffusion problem~\eqref{eq:mad-intro} with a priori error estimates. Numerical experiments are presented in Section~\ref{sec:num} followed by concluding remarks in Section~\ref{sec:conclusion}.

\section{Preliminaries}\label{sec:prelim}
\subsection{Notations}
    Throughout the presentation, we adhere to the convention that variables in plain style are scalar fields, while those in bold style are vector fields, that is, $\ub = \left[u_1, u_2, u_3\right]^\top$. The same rule applies to spaces (e.g., $\Lb^2$ stands for the vector $L^2$ space) and operators. We recall the relevant differential operators:
\begin{equation*}
    \grad: u\mapsto\begin{bmatrix}\partial_x u \\ \partial_y u \\ \partial_z u\end{bmatrix},\quad\curl: \ub\mapsto\begin{bmatrix}\partial_y u_3 - \partial_z u_2\\ \partial_z u_1 - \partial_x u_3 \\ \partial_x u_2 - \partial_y u_1 \end{bmatrix},\quad \div: \ub\mapsto \partial_x u_1 + \partial_y u_2 + \partial_z u_3.
\end{equation*}
The following notations are used for the trace operators on the boundary $\Gamma$:
\begin{equation}
    \gamma u := u, \quad \gammat \ub := \nb \times (\ub \times \nb), \quad \gammax \ub := \ub \times \nb, \quad \gamman \ub := \ub \cdot \nb
\end{equation}
representing the scalar trace, the tangential trace, the rotated tangential trace, and the normal trace, respectively.

For a bounded Lipschitz domain $D$, the standard Sobolev spaces are denoted by $H^1(D),\Hb(\boldsymbol{\mathrm{curl}},D)$ and $\Hb(\mathrm{div},D)$, which consist of functions with square-integrable derivatives, curls, and divergences, respectively. The space of functions with square-integrable derivatives up to order $r$ is denoted by $H^r(D), r\geq 0$. The convention $H^0(D) = L^2(D)$ is adopted. Sobolev norms and seminorms are denoted by $\|\cdot\|_{H^r(D)}$ and $|\cdot|_{H^r(D)}$, respectively. The spaces with vanishing traces are denoted by $H^1_0(D) := \{u \in H^1(D): u = 0 \text{ on }\partial D\}, \Hb_0(\boldsymbol{\mathrm{curl}},D) := \{\ub \in \Hb(\boldsymbol{\mathrm{curl}},D): \gammat \ub = 0 \text{ on }\partial D\}, \Hb_0(\mathrm{div},D):=\{\ub \in  \Hb(\mathrm{div},D):  \gamman \ub = 0\text{ on }\partial D\}$. To simplify the notation, the spatial domain may be omitted when $D = \Omega$, and we write $\|\cdot\|:= \|\cdot\|_{\Lb^2(\Omega)}$.

The curly bracket $(\cdot,\cdot)_D$ stands for $\Lb^2(D)$ inner product on a three-dimensional domain $D$, and the subscript may be omitted when $D = \Omega$. The angle bracket $\langle\cdot,\cdot\rangle_S$ stands for $\Lb^2(S)$ inner product on a two-dimensional surface $S$, and the subscript may be omitted when $S = \Gamma$. 

We use the symbol $C$ to represent a generic positive constant that may change from line to line. Importantly, it may depend on the spatial domain, the mesh shape regularity, and the polynomial degree but not on the mesh size.

\subsection{Finite element spaces}
Let $\{\mathcal{T}_h\}$ be a \emph{shape-regular} family of \emph{tetrahedral} triangulations of the polygonal domain $\Omega \subset \R^3$. Each tetrahedron/element $T \in \mathcal{T}_h$ is a closed domain with diameter $h_T$. The set of all faces of $\mathcal{T}_h$ is denoted by $\mathcal{F}_h$ while the set of (four) faces of $T$ is denoted by $\mathcal{F}_T$. We also adopt the following notations:
\begin{equation}\label{def:meshnotations}
    \begin{aligned}
        \Fcal_h^\Gamma &:= \{F \in \Fcal_h: F \subset \Gamma\} &&\text{ for faces on $\Gamma$},\\
        %\Fcal_h^- &:= \{F \in \Fcal_h: F \subset \Gamma^-\},\\
        \Fcal_h^\circ &:= \Fcal_h \setminus \Fcal_h^\Gamma &&\text{ for interior faces},\\
        \Tcal_h^\Gamma &:= \{T \in \Tcal_h: \Fcal_T \cap \Fcal_h^\Gamma \neq \emptyset\} &&\text{ for elements with at least one face on $\Gamma$}.
    \end{aligned}
\end{equation}
For the sake of presentation, we assume that the triangulation is \emph{quasi-uniform} throughout and denote $h := \max_{T \in \mathcal{T}_h} h_T$.

Let $k \in \{0, 1, 2, \ldots\}$ be a (fixed) polynomial degree. We write $\mathcal{P}_k$ for the space of $3$-variate polynomials of total degree $\leq k$ and $\boldsymbol{\mathcal{P}}_k = (\mathcal{P}_k)^3$. The following canonical finite element spaces in \emph{finite element exterior calculus} (FEEC)~\cite{arnold_2018} are considered:
\begin{equation}\label{def:spaces}
    \begin{aligned}
        P_h = P_h^k &:= \{v \in H^1:v|_T \in \mathcal{P}_{k+1}\;\;\forall\,T\in\mathcal{T}_h\}, \\
        \Sigmabh = \Sigmab^k_h &:= \{\vb \in \Hcurl: \vb|_T \in \boldsymbol{\mathcal{P}}_{k} + \xb\times\boldsymbol{\mathcal{P}}_{k}\;\;\forall\,T\in\mathcal{T}_h\},\\
        \Vb_{h} = \Vb_{h}^k &:= \{\vb \in \Hdiv: \vb|_T \in \boldsymbol{\mathcal{P}}_{k} + \xb\mathcal{P}_{k}\;\;\forall\,T\in\mathcal{T}_h\},  \\
        S_h = S_h^k &:= \{v \in L^2: v|_T \in \mathcal{P}_{k}\;\;\forall\,T\in\mathcal{T}_h\},
    \end{aligned}
\end{equation}
representing the Lagrange element, the \Nedelec element \emph{of the first kind}, the Raviart-Thomas element, and the discontinuous element of degree $k$. In the notation of FEEC~\cite{arnold_2018}, these spaces correspond to the trimmed polynomial spaces of differential forms $\mathcal{P}^{-}_{k+1}\Lambda^l(\mathcal{T}_h),\,l= 0,1,2,3$, respectively. 

With $\omega \subset \Omega$ a closed subdomain that is a union of elements, we denote 
\begin{equation}
    \Sigmabh(\omega) := \{\vb \in \Hb(\curln,\omega): \vb|_T \in \boldsymbol{\mathcal{P}}_{k} + \xb\times\boldsymbol{\mathcal{P}}_{k}\;\;\forall\,T\subset\omega\}
\end{equation}
as the local \Nedelec element space on $\omega$. Denoting the $\Gamma$-boundary of $\omega$ as
\begin{equation}
    \Gamma_\omega := \partial \omega \cap \Gamma,
\end{equation}
we write 
\begin{equation}
    \Sigmab_{h,0}(\omega) := \{\vb \in \Sigmabh(\omega): \gammat \vb = \zerob \text{ on } \Gamma_\omega\}
\end{equation} 
as the set of functions in $\Sigmabh(\omega)$ with vanishing (tangential) trace on $\Gamma_\omega$.
The same notation with proper traces applies to $P_h(\omega), \Vb_h(\omega)$ and $P_{h,0}(\omega), \Vb_{h,0}(\omega)$. If $\omega$ is of trivial topology, we have the local exact sequence~\cite[][Chapter~7]{arnold_2018}:
\begin{equation}\label{eq:exseqbulk}
    \begin{tikzcd}
        P_{h,0}(\omega) \arrow[r, "\grad"] & \Sigmab_{h,0}(\omega) \arrow[r, "\curl"] & \Vb_{h,0}(\omega) \arrow[r, "\div"] & S_{h}(\omega).
    \end{tikzcd}
\end{equation}
That is, $\grad P_{h,0}(\omega) \subset \Sigmab_{h,0}(\omega)$ coincides with the kernel of $\curl: \Sigmab_{h,0}(\omega) \rightarrow \Vb_{h,0}(\omega)$ and $\curl \Sigmab_{h,0}(\omega) \subset \Vb_{h,0}(\omega)$ coincides with the kernel of $\div: \Vb_{h,0}(\omega) \rightarrow S_h(\omega)$.
 
Additionally, we shall need the discrete trace spaces defined on the surface $\Gamma_\omega$. Denote
\begin{equation}
    P_h(\Gamma_\omega) := \gamma P_h(\omega), \quad \Sigmab_h(\Gamma_\omega) := \gammat \Sigmab_h(\omega), \quad S_h(\Gamma_\omega) := \gamman \Vb_h(\omega).
\end{equation}
If $\Gamma_\omega$ is of trivial topology, we have the exact sequence~\cite[][Chapter~7]{arnold_2018}:
\begin{equation}\label{eq:exseqsurf}
    \begin{tikzcd}
        P_h(\Gamma_\omega) \arrow[r, "\gradGamma"] & \Sigmab_h(\Gamma_\omega) \arrow[r, "\curlGamma"] & S_h(\Gamma_\omega)
    \end{tikzcd}
\end{equation}
where $\gradGamma$ and $\curlGamma$ are the \emph{surface gradient} and \emph{surface curl}, respectively. That is, $\gradGamma P_h(\Gamma_\omega) \subset \Sigmab_h(\Gamma_\omega)$ coincides with the kernel of $\curlGamma: \Sigmab_h(\Gamma_\omega) \rightarrow S_h(\Gamma_\omega)$.

We mention some standard inverse inequalities assuming the shape-regularity of $\{\Tcal_h\}$.
\begin{lemma}[inverse inequalities\text{~\cite[][Lemma~12.1 and~12.8]{ern_2021}}]\label{thm:invT}
    Let $\boldsymbol{\Pcal}(T)$ be a finite-dimensional subspace of $\Hb^1(T)$.
    There exist constants $C_{\mathrm{inv}}, C_{\mathrm{tr}} > 0$ independent of $h$ such that 
    \begin{align}
        |\vb|_{\Hb^1(T)} &\leq C_{\mathrm{inv}} h_T^{-1} \|\vb\|_{\Lb^2(T)} , \label{eq:inv}\\
        \|\vb\|_{\Lb^2(F_T)} &\leq C_{\mathrm{tr}} h_T^{-1/2} \|\vb\|_{\Lb^2(T)}, \label{eq:invtrace}
    \end{align}
    for any $\vb \in \boldsymbol{\Pcal}(T), T \in \Tcal_h, F \in \Fcal_T$.
\end{lemma}

\subsection{Interpolation}\label{sec:IND}
Let $\INDav: \Lb^2 \rightarrow \Sigmab_h$ be an interpolation operator satisfying the following approximation properties:
\begin{equation}\label{eq:approxIND}
    \begin{aligned}
        |\vb - \INDav \vb|_{\Hb^m} &\leq C h^{r-m} |\vb|_{\Hb^r}, \\
        |\curl (\vb - \INDav \vb)|_{\Hb^m} &\leq C h^{r-m} |\curl \vb|_{\Hb^r} \\
    \end{aligned}
\end{equation}
for $m \in [0, r],\; r\in [0, k+1]$. In addition, we require that $\INDav$ preserves homogeneous boundary conditions, that is, $\INDav \vb \in \Sigmab_{h,0}$ if $\vb \in \Hzcurl$. Interpolations satisfying these properties include the so-called \emph{$L^2$-bounded commuting projections} in the FEEC framework and can be found in, e.g.,~\cite[][Chapter~23]{ern_2021}.

Recall also the multiplicative trace inequality~\cite[][Lemma~12.15]{ern_2021}:
\begin{equation}\label{eq:multtraceineq}
    \|\vb\|^2_{\Lb^2(F_T)} \leq C\left(h_T^{-1}\|\vb\|^2_{\Lb^2(T)} + h_T|\vb|^2_{\Hb^1(T)}\right),
\end{equation}
for any $F_T \in \Fcal_T, T \in \Tcal_h$. 
Consequently, the following trace approximation estimates hold:
\begin{equation}\label{eq:traceapproxIND}
    \begin{aligned}
    \left(\sum_{F \in \Fcal_h}\|\vb - \INDav \vb\|^2_{\Lb^2(F)}\right)^{1/2} &\leq C h^{r-1/2} |\vb|_{\Hb^r} , \\
    \left(\sum_{F \in \Fcal_h}\|\curl (\vb - \INDav \vb)\|^2_{\Lb^2(F)}\right)^{1/2} &\leq C h^{r-1/2} |\curl \vb|_{\Hb^r} ,
    \end{aligned}
\end{equation}
for $r\in [1,k+1]$.

\section{Inf-sup stability of the asymmetric bilinear form}\label{sec:infsup}
\subsection{Review of the scalar case}\label{sec:scalar}
Consider the Poisson equation with homogeneous Dirichlet BC:
\begin{equation}
    \begin{aligned}
        -\div \grad u &= f && \text{in }\Omega,\\
        u &= 0 && \text{on }\Gamma.
    \end{aligned}
\end{equation}
The symmetric Nitsche's method~\cite{nitsche_1971} seeks $u_h \in P_{h}$ such that
\begin{equation}\label{eq:blsymscalar}
    \a^\mathrm{sym}_{\gradn,h}(u_h,v_h) := (\grad u_h, \grad v_h) - \langle \partial_n u_h, v_h \rangle - \langle \partial_n v_h, u_h \rangle + \frac{C_p}{h}\langle u_h, v_h \rangle = (f,v_h) \quad \forall\,v_h\in P_{h}
\end{equation}
where $\partial_n u := \grad u \cdot \nb$ is the normal derivative of $u$ and $C_p > 0$ is a penalty parameter that must be chosen sufficiently large.
Introducing the $h$-dependent norm $\|v\|_{\gradn,h}^2 := \|\grad v\|^2 + h^{-1}\|v\|^2_{\Lb^2(\Gamma)}$, it is standard to show the coercivity of $\a^\mathrm{sym}_{\gradn,h}$ with respect to $\|\cdot\|_{\gradn,h}$ when $C_p$ is sufficiently large. On the other hand, the penalty-free asymmetric Nitsche's method seeks $u_h \in P_{h}$ such that 
\begin{equation}\label{eq:blasymscalar}
    \a^\mathrm{asym}_{\gradn,h}(u_h,v_h) := (\grad u_h, \grad v_h) - \langle \partial_n u_h, v_h \rangle + \langle \partial_n v_h, u_h \rangle = (f,v_h) \quad \forall\,v_h\in P_{h}.
\end{equation}
The bilinear form $\a^\mathrm{asym}_{\gradn,h}$ is only coercive with respect to the seminorm $\|\grad u_h \|$, which does not guarantee the stability of the method. In fact, the closely related \emph{nonsymmetric interior penalty discontinuous Galerkin} (NIPG) method is known to be unstable using piecewise linear elements~\cite{arnold_2002}. The $H^1$-conforming case is different. It can be shown that $\a^\mathrm{asym}_{\gradn,h}$ satisfies the following inf-sup stability estimate:
\begin{equation}\label{eq:infsupasymscalar}
    \sup_{v_h \in P_h} \frac{\a^\mathrm{asym}_{\gradn,h}(u_h,v_h)}{\|v_h\|_{\gradn,h}} \geq \beta \|u_h\|_{\gradn,h} \quad \forall\,u_h\in P_h
\end{equation}
under a mild mesh condition~\cite{burman_2012}. The key step is to construct a test function $v_h$ to control certain averages of the solution on the boundary via the boundary term $\langle \partial_n v_h, u_h \rangle$. Our result can be viewed as the $H(\curl)$-counterpart of~\eqref{eq:infsupasymscalar}. Yet the construction of the test function is more involved due to the vector-valued nature of the problem and the large kernel of the $\curln$ operator.

\subsection{Isolated patch condition}\label{sec:patch}
As in the scalar case~\cite[][Section~4]{burman_2012}, certain topological condition on the family of meshes $\{\Tcal_h\}$ is required. We suppose that there exists a \emph{patch decomposition} $\Wcal_h = \{\omega\}$ of $\Tcal_h^\Gamma$ (see definition in~\eqref{def:meshnotations}) such that $\bigcup_{\omega \in \Wcal_h} \omega = \Tcal_h^\Gamma$ and the following conditions are fulfilled:
\begin{enumerate}[label=(H\arabic*)]
    \item each patch $\omega \in \Wcal_h$ contains a \emph{bounded} number of elements and is \emph{simply-connected};\label{ass:bddnum}
    \item $\Gamma_\omega (:= \partial\omega \cap \Gamma)$ is \emph{co-planar} for each $\omega \in \Wcal_h$ and is \emph{simply-connected};\label{ass:coplanar}
    \item each patch $\omega \in \Wcal_h$ does not share any \emph{interior faces and edges} with another patch $\omega'$, that is, $(\partial\omega \cap \partial\omega')\setminus \Gamma = \emptyset$ for $\omega \neq \omega'$. \label{ass:isolated}
\end{enumerate}
We call the collection of~\ref{ass:bddnum}-\ref{ass:isolated} the \emph{isolated patch condition}. A sketch is shown in Figure~\ref{fig:patch}. It is worth noting that restrictiveness of the isolated patch condition comes mostly from~\ref{ass:isolated}. For instance, it rules out meshes with tetrahedra that have two or more faces on $\Gamma$. Nevertheless, meshes fulfilling the isolated patch condition can be generated via suitable refinement of tetrahedra with multiple faces on $\Gamma$. Furthermore, numerical experiments (see Section~\ref{sec:num}) indicate a fairly high tolerance to violation of these conditions. In particular, we observe that the method performs well even when the mesh contains tetrahedra with two faces on $\Gamma$. Therefore, we anticipate that the isolated patch condition can be weakened to obtain the same stability result (Lemma~\ref{thm:infsupasym}).
    \begin{figure}[h]
        \centering
\begin{tikzpicture}[scale=0.7,line join=round,line cap=round]

\tikzset{
  surf/.style   ={fill=gray!95, fill opacity=0.34, draw=gray, draw opacity=0.95, thin},
  patchA/.style ={fill=blue!75, fill opacity=0.30, draw=blue!70, very thick},
  patchB/.style ={fill=red!75,  fill opacity=0.30, draw=red!70,  very thick},
  tetA/.style   ={fill=blue!75, fill opacity=0.20, draw=blue!80!black, thin},
  tetB/.style   ={fill=red!75,  fill opacity=0.18, draw=red!80!black,  thin},
  ed/.style     ={draw=gray!70, thin},
  hidden/.style ={draw=gray!65, dashed, thin},
  every node/.style={font=\small}
}

% -------------------------------------------------
% planar vertices of the boundary triangulation
% slightly perturbed to look more natural
% -------------------------------------------------
\coordinate (p2) at (3.12,2.61);
\coordinate (p3) at (4.18,3.34);
\coordinate (p4) at (5.58,2.01);
\coordinate (p5) at (6.96,3.63);
\coordinate (p6) at (8.47,2.01);
\coordinate (p7) at (9.46,3.11);
\coordinate (p8) at (10.28,2.24);

% one-ring outer vertices, also perturbed
\coordinate (q0) at (1.58,3.4);
\coordinate (q1) at (4.96,4.42);
\coordinate (q2) at (8.08,4.10);
\coordinate (q3) at (10.48,3.76);

\coordinate (r0) at (1.92,1.58);
\coordinate (r1) at (5.02,0.59);
\coordinate (r2) at (11.18,1.21);
\coordinate (r3) at (11.68,3.04);

% interior vertices of tetrahedra
\coordinate (vL) at (5.50,0.02);
\coordinate (u)  at (7.98,-0.02);

% -------------------------------------------------
% 1) two patches first
% -------------------------------------------------
% blue patch
\filldraw[patchA] (p2)--(p3)--(p4)--cycle; % F1
\filldraw[patchA] (p3)--(p4)--(p5)--cycle; % F2

% red patch
\filldraw[patchB] (p4)--(p5)--(p6)--cycle; % F3
\filldraw[patchB] (p5)--(p6)--(p7)--cycle; % F4
\filldraw[patchB] (p6)--(p7)--(p8)--cycle; % F5

% -------------------------------------------------
% 2) planar boundary surface mesh around them
% -------------------------------------------------
% upper ring
\filldraw[surf] (q0)--(p2)--(p3)--cycle;
\filldraw[surf] (q0)--(p3)--(q1)--cycle;
\filldraw[surf] (q1)--(p3)--(p5)--cycle;
\filldraw[surf] (q1)--(p5)--(q2)--cycle;
\filldraw[surf] (q2)--(p5)--(p7)--cycle;
\filldraw[surf] (q2)--(p7)--(q3)--cycle;
\filldraw[surf] (q3)--(p7)--(p8)--cycle;

% lower ring
\filldraw[surf] (r0)--(p2)--(p4)--cycle;
\filldraw[surf] (r0)--(p4)--(r1)--cycle;
\filldraw[surf] (r1)--(p4)--(p6)--cycle;
\filldraw[surf] (r1)--(p6)--(r2)--cycle;
\filldraw[surf] (r2)--(p6)--(p8)--cycle;
\filldraw[surf] (r2)--(p8)--(r3)--cycle;

% left/right closures
\filldraw[surf] (q0)--(r0)--(p2)--cycle;
\filldraw[surf] (q3)--(p8)--(r3)--cycle;

% -------------------------------------------------
% 3) tetrahedra under the two patches
% -------------------------------------------------
% blue patch tetrahedra
% T1=(p2,p3,p4,vL), T2=(p3,p4,p5,vL)
\filldraw[tetA] (p2)--(p3)--(vL)--cycle;
\filldraw[tetA] (p2)--(p4)--(vL)--cycle;
\filldraw[tetA] (p3)--(p4)--(vL)--cycle;

\filldraw[tetA] (p3)--(p5)--(vL)--cycle;
\filldraw[tetA] (p4)--(p5)--(vL)--cycle;
\filldraw[tetA] (p3)--(p4)--(vL)--cycle;

% red patch tetrahedra
% T3=(p4,p5,p6,u), T4=(p5,p6,p7,u), T5=(p6,p7,p8,u)
\filldraw[tetB] (p4)--(p5)--(u)--cycle;
\filldraw[tetB] (p4)--(p6)--(u)--cycle;
\filldraw[tetB] (p5)--(p6)--(u)--cycle;

\filldraw[tetB] (p5)--(p6)--(u)--cycle;
\filldraw[tetB] (p5)--(p7)--(u)--cycle;
\filldraw[tetB] (p6)--(p7)--(u)--cycle;

\filldraw[tetB] (p6)--(p7)--(u)--cycle;
\filldraw[tetB] (p6)--(p8)--(u)--cycle;
\filldraw[tetB] (p7)--(p8)--(u)--cycle;

% emphasize shared interior faces in red patch
\draw[red!90!black, line width=0.9pt] (p5)--(p6)--(u)--cycle;
\draw[red!90!black, line width=0.9pt] (p6)--(p7)--(u)--cycle;

% -------------------------------------------------
% mesh edges
% -------------------------------------------------
\draw[ed] (q0)--(q1)--(q2)--(q3);
\draw[ed] (p2)--(p3)--(p4)--(p5)--(p6)--(p7)--(p8);
\draw[ed] (r0)--(r1)--(r2)--(r3);

\draw[ed] (q0)--(p2) (q0)--(p3)
          (q1)--(p3) (q1)--(p5)
          (q2)--(p5) (q2)--(p7)
          (q3)--(p7) (q3)--(p8);

\draw[ed] (r0)--(p2) (r0)--(p4)
          (r1)--(p4) (r1)--(p6)
          (r2)--(p6) (r2)--(p8)
          (r3)--(p8);

\draw[ed] (q0)--(r0) (q3)--(r3);

% patch outlines
\draw[blue!85] (p2)--(p3)--(p4)--cycle;
\draw[blue!85] (p3)--(p4)--(p5)--cycle;

\draw[red!85]  (p4)--(p5)--(p6)--cycle;
\draw[red!85]  (p5)--(p6)--(p7)--cycle;
\draw[red!85]  (p6)--(p7)--(p8)--cycle;

% shared boundary edge between the two patches
\draw[line width=1.0pt] (p4)--(p5);

% interior edges
\draw[blue!80!black] (p2)--(vL) (p3)--(vL) (p4)--(vL) (p5)--(vL);
\draw[red!80!black]  (p4)--(u) (p5)--(u) (p6)--(u) (p7)--(u) (p8)--(u);

% dashed hints into Omega
\draw[hidden] (vL)--++(-0.9,-1.15);
\draw[hidden] (vL)--++( 0.2,-1.25);
\draw[hidden] (u)--++(-0.3,-1.20);
\draw[hidden] (u)--++( 0.7,-1.05);

% labels
\node at (3.15,3.82) {\huge\(\Gamma\)};
\node at (6.5,-0.55) {\huge\(\Omega\)};

\node[blue!80!black] at (5.00,2.78) {\Large\(\Gamma_{\omega}\)};

\node[red!80!black] at (8.52,2.74) {\Large\(\Gamma_{\omega'}\)};

\node[blue!80!black] at (5.08,1.28) {\Large\(\omega\)};

\node[red!80!black] at (7.92,1.28) {\Large\(\omega'\)};

\fill (vL) circle (1.2pt);
\fill (u)  circle (1.2pt);

\end{tikzpicture}

        \caption{A sketch of the isolated patch condition. The blue and red patches $\omega$ and $\omega'$ contain two and three elements, respectively. }
        \label{fig:patch}
\end{figure}
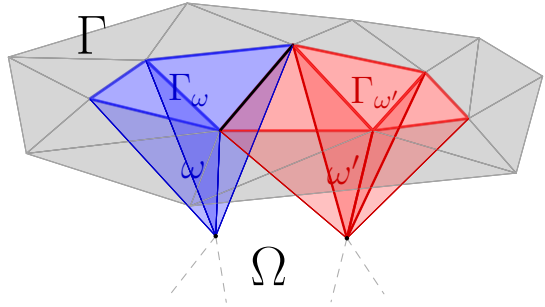

\subsection{Inf-sup estimate}
Introduce the $h$-dependent norms
\begin{equation}\label{def:hseminorm}
\begin{aligned}
    |\vb|_{\curln,h}^2 &:= \|\curl\vb\|^2 + h^{-1}\|\gammat\vb\|_{\Lb^2(\Gamma)}^2, \\
    \|\vb\|_{\curln,h}^2 &:= \|\vb\|^2 + |\vb|_{\curln,h}^2.
\end{aligned}
\end{equation}
Recall that 
\begin{equation}\label{eq:blasym}
    \a^\mathrm{asym}_{\curln,h}(\ub_h,\vb_h) := (\curl\ub_h,\curl\vb_h)- \langle \gammax \curl \ub_h, \gammat \vb_h \rangle + \langle \gammax \curl \vb_h, \gammat \ub_h  \rangle.
\end{equation}
The main challenge in establishing the inf-sup estimate is to construct a test function capable of controlling the $\Lb^2$ norm of $\gammat \ub_h$ via the bilinear form. 

First, some notations are introduced. For each patch $\omega \in \Wcal_h$ complying with the isolated patch condition (see Section~\ref{sec:patch}), recall the exact sequences~\eqref{eq:exseqbulk} on $\omega$ and~\eqref{eq:exseqsurf} on $\Gamma_\omega$. 
Denote the kernel space of $\curlGamma : \Sigmab_h(\Gamma_\omega) \rightarrow S_h(\Gamma_\omega)$ by 
\begin{equation}\label{def:Mb}
    \Mb^\Gamma_\omega := \gradGamma P_h(\Gamma_\omega) \subset \Sigmab_h(\Gamma_\omega),
\end{equation}
and the $\Lb^2$-orthogonal projection to $\Mb^\Gamma_\omega$ by $\Pib_\omega^\Gamma: \Lb^2(\Gamma_\omega) \rightarrow \Mb^\Gamma_\omega$. Collecting the local projections and spaces, we define 
\begin{equation}\label{def:PiMb}
    \Pib^\Gamma_h := \prod_{\omega \in \Wcal_h} \Pib_\omega^\Gamma, \quad \quad  \Mb^\Gamma_h := \prod_{\omega \in \Wcal_h} \Mb^\Gamma_\omega,
\end{equation}
with the following interpretation:
\begin{equation}\label{eq:l2int}
    \|\Pib^\Gamma_h \gb^\Gamma\|_{\Lb^2(\Gamma)}^2 = \sum_{\omega\in\Wcal_h} \|\Pib_\omega^\Gamma \gb^\Gamma\|_{\Lb^2(\Gamma_\omega)}^2 \quad \text{for}\,\,\gb^\Gamma \in \Lb^2(\Gamma).
\end{equation}

Next, we provide an auxiliary lemma deduced from an elementary scaling argument.
\begin{lemma}[patch-wise scaling argument]\label{thm:scalingarg}
    There exists a constant $C > 0$ independent of $h$ such that
    \begin{equation}\label{eq:scalingarg}
        \|\grad \psi_h\|_{\Lb^2(\omega)} \leq C h^{1/2} \|\gradGamma \psi_h\|_{\Lb^2(\Gamma_\omega)}\quad\forall\,\omega \in \Wcal_h, \psi_h \in P_{h,\Gamma}(\omega)
    \end{equation}
    where
    \begin{equation}\label{def:auxspace}
        P_{h,\Gamma}(\omega) := \{v_h \in P_h(\omega): \text{\emph{all DoFs not lying on $\Gamma_\omega$ vanish and }} \langle v_h, 1\rangle_{\Gamma_\omega} = 0\}.
    \end{equation}
\end{lemma}
\begin{proof}
    Note that both $\omega$ and $\Gamma_\omega$ are simply-connected from Assumption~\ref{ass:bddnum} and~\ref{ass:coplanar}. We first check that both $\|\grad \psi_h\|_{\Lb^2(\omega)}$ and $\|\gradGamma \psi_h\|_{\Lb^2(\Gamma_\omega)}$ define a norm on the finite dimensional space $P_{h,\Gamma}(\omega)$. If $\|\grad \psi_h\|_{\Lb^2(\omega)} = 0$, then $\psi_h$ is constant on $\omega$, and since $\langle \psi_h, 1\rangle_{\Gamma_\omega} = 0$, we have $\psi_h = 0$. Similarly, if $\|\gradGamma \psi_h\|_{\Lb^2(\Gamma_\omega)} = 0$, then $\psi_h$ is constant on $\Gamma_\omega$, and since $\langle \psi_h, 1\rangle_{\Gamma_\omega} = 0$, we have $\psi_h = 0$ on $\Gamma_\omega$. Since the only nonzero DoFs are those on the boundary, we have $\psi_h = 0$ on $\Gamma$. Thus, both norms are indeed norms on $P_{h,\Gamma}(\omega)$. The estimate~\eqref{eq:scalingarg} then follows from the equivalence of norms on finite dimensional spaces and the scaling argument based on the assumption that $\mathrm{diam}(\omega) \sim h$ from~\ref{ass:bddnum} and $\Tcal_h$ is shape-regular.
\end{proof}
The following crucial lemma constructs the desired test function that controls $\|\Pib^\Gamma_h \gammat \ub_h\|_{\Lb^2(\Gamma)}$ via the boundary term $\langle \gammax \curl \vb_h, \gammat \ub_h  \rangle$ in~\eqref{eq:blasym}.
\begin{lemma}[a lifting operator]\label{thm:lift}
Suppose that the isolated patch condition in Section~\ref{sec:patch} holds. There exists a linear operator $\Rsfb_h: \Mb^\Gamma_h \rightarrow \Sigmab_{h,0}$ such that
\begin{equation}\label{eq:bddcontrol}
    \gammax \curl \Rsfb_h \gb^\Gamma_h = \gb^\Gamma_h \quad \forall\,\gb^\Gamma_h \in \Mb^\Gamma_h.
\end{equation}
Moreover, there exists a constant $C > 0$ independent of $h$ such that
\begin{align}
    \|\Rsfb_h \gb^\Gamma_h\|_{\Lb^2(\Omega)} &\leq Ch^{3/2}\|\gb^\Gamma_h\|_{\Lb^2(\Gamma)} \quad \forall\,\gb^\Gamma_h \in \Mb^\Gamma_h.\label{eq:bddtestfuncbound}
\end{align}
\end{lemma}
\begin{proof}
    Given $\gb^\Gamma_h \in \Mb^\Gamma_h$, we construct $\Rsfb_h \gb^\Gamma_h$ patch by patch. For each patch $\omega \in \Wcal_h$, denote $\gb_\omega^\Gamma := \gb^\Gamma_h|_{\Gamma_\omega} \in \Mb_\omega^\Gamma$. By the definition of $\Mb_\omega^\Gamma$ in~\eqref{def:Mb}, there exists $\psi^\Gamma_\omega \in P_h(\Gamma_\omega)$ such that $\gradGamma \psi^\Gamma_\omega = \gb_\omega^\Gamma$ and $\langle \psi^\Gamma_\omega, 1\rangle_{\Gamma_\omega} = 0$.
    Let $\psi_\omega \in P_h(\omega)$ denote the extension of $\psi^\Gamma_\omega$ to $\omega$ by keeping boundary DoFs while setting interior DoFs to zero. 
    In other words, $\psi_\omega \in P_{h,\Gamma}(\omega)$ (see definition in Lemma~\ref{thm:scalingarg}).
    It is easy to see that $\gradGamma \psi^\Gamma_\omega = \gammat \grad \psi_\omega$. 
    Note that the outward normal vector $\nb$ on $\Gamma_\omega$ is a constant vector due to the co-planar assumption~\ref{ass:coplanar}. 
    By vector calculus, we have $\curl (\psi_\omega \nb) = \grad \psi_\omega \times \nb$. Hence, $\gammax\curl (\psi_\omega \nb) = - \gammat\grad\psi_\omega = -\gradGamma \psi^\Gamma_\omega = -\gb^\Gamma_\omega$. 
    It holds that $\curl (\psi_\omega \nb) \in \Vb_{h,0}(\omega)$ since $\psi_\omega$ is polynomial of degree at most $k+1$. Due to the exact sequence~\eqref{eq:exseqbulk}, we can find $\wb_\omega \in \Sigmab_{h,0}(\omega) \cap (\grad P_{h,0}(\omega))^\perp$ such that $\curl \wb_\omega = \curl (\psi_\omega \nb)$ and 
    \begin{equation}\label{eq:discpoincare}
        \|\wb_\omega\|_{\Lb^2(\omega)} \leq Ch\|\curl (\psi_\omega \nb)\|_{\Lb^2(\omega)}
    \end{equation}
    by the discrete \Poincare inequality~\cite[][Theorem~5.3]{arnold_2018}. The $h$-scaling is due to $\mathrm{diam}(\omega)\sim h$ from Assumption~\ref{ass:bddnum}.
    
    Let $\widetilde{\wb}_\omega \in \Sigmab_{h,0}$ be a \Nedelec function defined globally such that its DoFs coincide with those of $\wb_\omega$ on all mesh entities belonging to $\omega$, and are zero otherwise. This extension is $\Lb^2$-bounded, that is, $\|\widetilde{\wb}_\omega\|_{\Lb^2(\Omega)} \leq C\|\wb_\omega\|_{\Lb^2(\omega)}$. Since $\wb_\omega$ has vanishing tangential trace on $\Gamma_\omega$, all boundary DoFs are zero. By the isolated condition~\ref{ass:isolated}, two distinct patches $\omega_1, \omega_2$ share no interior mesh entities. Together with the fact that $\wb_\omega$ has zero DoFs on $\Gamma$, we see that $\widetilde{\wb}_{\omega_1}$ has zero DoFs on all entities associated with $\omega_2$. This yields \emph{patch-wise decoupling}.
    
    We construct $\varphib_h := \Rsfb_h \gb^\Gamma_h \in \Sigmab_{h,0}$ by taking the summation of $\{\widetilde{\wb}_\omega\}_{\omega \in \Wcal_h}$. That is,
    \begin{equation}
        \varphib_h = -\sum_{\omega\in\Wcal_h} \widetilde{\wb}_\omega.
    \end{equation}
    The patch-wise decoupling implies that $\curl\varphib_h|_{\omega} = -\curl \widetilde{\wb}_{\omega}$ and thus 
    \begin{equation}
        \gammax \curl \varphib_h|_{\Gamma_\omega} = -\gammax\curl{\widetilde{\wb}_\omega} = -\gammax\curl{\wb_\omega} = -\gammax\curl{(\psi_\omega\nb)} = \gb_\omega^\Gamma.
    \end{equation}
    This property leads to~\eqref{eq:bddcontrol}.
    To show the global bound~\eqref{eq:bddtestfuncbound}, we first deduce local bounds:
    \begin{equation}\label{eq:localbounds}
        \begin{aligned}
            \|\curl \wb_\omega\|_{\Lb^2(\omega)} &= \|\curl (\psi_\omega \nb)\|_{\Lb^2(\omega)} = \|\grad \psi_\omega \times \nb_\omega\|_{\Lb^2(\omega)} \\
            &\leq \|\grad \psi_\omega\|_{\Lb^2(\omega)} \leq Ch^{1/2}\|\gradGamma \psi_\omega\|_{\Lb^2(\Gamma_\omega)} = Ch^{1/2}\|\gb_\omega^\Gamma\|_{\Lb^2(\Gamma_\omega)}. 
        \end{aligned}
    \end{equation}
    Here we used the estimate~\eqref{eq:scalingarg}.
    Combining~\eqref{eq:discpoincare} and~\eqref{eq:localbounds}, we have
    \begin{equation}
        \|\wb_\omega\|_{\Lb^2(\omega)} \leq Ch\|\curl \wb_\omega\|_{\Lb^2(\omega)} \leq Ch^{3/2} \|\gb_\omega^\Gamma\|_{\Lb^2(\Gamma_\omega)}.
    \end{equation}
    Since the extension is stable in $\Lb^2$, it holds that
        $\|\widetilde{\wb}_\omega\|_{\Lb^2(\Omega)} \leq C\|\wb_\omega\|_{\Lb^2(\omega)}$.
    Due to finite overlap of $\{\supp(\widetilde{\wb}_\omega)\}$, we collect local contributions and obtain
    \begin{equation}
        \begin{aligned}
            \|\varphib_h\|^2 &= \left\|\sum_{\omega\in\Wcal_h} \widetilde{\wb}_{\omega} \right\|^2 \leq C\sum_{\omega\in\Wcal_h} \|\widetilde{\wb}_\omega\|^2_{\Lb^2(\Omega)} \leq C\sum_{\omega\in\Wcal_h} \|\wb_\omega\|^2_{\Lb^2(\omega)} \leq C\sum_{\omega\in\Wcal_h} h^{3} \|\gb_\omega^\Gamma\|_{\Lb^2(\Gamma_\omega)}^2.
        \end{aligned}
    \end{equation}
    In view of~\eqref{eq:l2int}, we can conclude.
\end{proof}

Using the lifting operator $\Rsfb_h$ in Lemma~\ref{thm:lift} and letting $\vb_h = h^{-1}\Rsfb_h \Pib_h^\Gamma\gammat\ub_h$, we have $\langle \gammax \curl \vb_h, \gammat \ub_h  \rangle = h^{-1}\|\Pib^\Gamma_h \gammat \ub_h\|^2_{\Lb^2(\Gamma)}$. This is not yet sufficient to control the full $\Lb^2$ norm $h^{-1/2}\|\gammat \ub_h\|_{\Lb^2(\Gamma)}$. The next lemma states a discrete Poincar\'e-type inequality implying that the remaining part can be controlled by $\|\curl \ub_h\|$. 
\begin{lemma}[trace \Poincare inequality]\label{thm:tracepoincare}
Suppose that the isolated patch condition in Section~\ref{sec:patch} holds. There exists a constant $C > 0$ independent of $h$ such that for each patch $\omega \in \Wcal_h$ and $\ub_h \in \Sigmab_h$, it holds that
\begin{equation}\label{eq:tracepoincare}
    \|\gammat\ub_h - \Pib_h^\Gamma \gammat\ub_h\|_{\Lb^2(\Gamma_\omega)} \leq Ch^{1/2}\|\curl \ub_h\|_{\Lb^2(\omega)}. 
\end{equation}
\end{lemma}
\begin{proof}
    For each patch $\omega$, recall the exact sequence~\eqref{eq:exseqsurf}. For any $\ub_h\in \Sigmab_h$, we have the $\Lb^2$-orthogonal decomposition $\gammat\ub_h = \Pib_h^\Gamma \gammat\ub_h + (\gammat\ub_h - \Pib_h^\Gamma \gammat\ub_h)$. The second term $(\gammat\ub_h - \Pib_h^\Gamma\gammat \ub_h)$ is $\Gamma_\omega$-piecewise $\Lb^2$-orthogonal to $\Mb_\omega^\Gamma (:= \gradGamma P_h(\Gamma_\omega))$. That is, 
    \begin{equation}
        \langle\gammat\ub_h - \Pib_h^\Gamma \gammat\ub_h, \taub^\Gamma_{h,\omega}\rangle_{\Gamma_\omega} = 0 \quad \forall\,\taub^\Gamma_{h,\omega} \in \Mb_\omega^\Gamma.
    \end{equation}
    Then, we apply the discrete Poincar\'e inequality~\cite[][Theorem~5.3]{arnold_2018} on each $\Gamma_\omega$:
    \begin{equation}
        \|\gammat\ub_h - \Pib_h^\Gamma \gammat\ub_h\|_{\Lb^2(\Gamma_\omega)} \leq Ch\|\curlGamma (\gammat\ub_h - \Pib_h^\Gamma \gammat\ub_h)\|_{\Lb^2(\Gamma_\omega)} = Ch\|\curlGamma \gammat\ub_h\|_{\Lb^2(\Gamma_\omega)}.
    \end{equation} 
    The $h$-scaling is due to the fact that $\mathrm{diam}(\Gamma_\omega) \sim h$ from the assumption~\ref{ass:bddnum} and scaling arguments. Note that $\curlGamma \gammat\ub_h = \gamman\curl \ub_h$. By the trace inequality $\|\curlGamma \gammat\ub_h\|_{\Lb^2(\Gamma_\omega)} \leq Ch^{-1/2}\|\curl \ub_h\|_{\Lb^2(\omega)}$, we can conclude.
\end{proof}
\noindent The following seminorm equivalence is a direct consequence of Lemma~\ref{thm:tracepoincare}.
\begin{corollary}[seminorm equivalence]\label{thm:tracenormequiv}
    Suppose that the isolated patch condition in Section~\ref{sec:patch} holds. There exist constants $C_1, C_2 > 0$ independent of $h$ such that
    \begin{equation}\label{eq:tracenormequiv}
        C_1|\ub_h|^2_{\curln,h} \leq \|\curl \ub_h\|^2 + h^{-1}\|\Pib_h^\Gamma \gammat \ub_h\|^2_{\Lb^2(\Gamma)} \leq C_2|\ub_h|_{\curln,h}^2 \quad \forall\,\ub_h \in \Sigmab_h.
    \end{equation}
\end{corollary}

Lastly, we shall need the following discrete Poincar\'e-type inequality.
\begin{lemma}[a discrete Poincar\'e inequality]\label{thm:discpoincareaux}
    There exists a constant $C > 0$ independent of $h$ such that: for each $\vb_h \in \Sigmab_h$, there exists $r_h \in P_{h,0}$ linearly dependent on $\vb_h$ and
    \begin{equation}\label{eq:discpoincareaux}
        \|\vb_h  - \grad r_h\| \leq C |\vb_h|_{\curln,h}.
    \end{equation}
\end{lemma}
\begin{proof}
    Denote $\vb^\partial_h \in \Sigmab_h$ as the function that shares the same DoFs as $\vb_h$ on $\Gamma$ and vanishes on interior DoFs. Consider the splitting $\vb_h = \vb_{h}^\circ + \vb_{h}^\partial$ where $\vb_{h}^\circ \in \Sigmab_{h,0}$. By a simple scaling argument, one can verify that 
    \begin{equation}\label{eq:cc:bdryscale}
        \|\vb_h^\partial\| \leq C_{\mathrm{1}} h^{1/2} \|\gammat\vb_h\|_{\Lb^2(\Gamma)}.
    \end{equation}
    Recall the discrete Hodge decomposition of $\vb_h^\circ \in \Sigmab_{h,0}$~\cite[][Eq.~5.6]{arnold_2018}:
    \begin{equation}\label{eq:cc:hodge}
        \vb_h^\circ = \grad r_h + \mub_h, \quad r_h \in P_{h,0},\, \mub_h \in \Sigmab_{h,0} \cap (\grad P_{h,0})^\perp,
    \end{equation}
    and the discrete \Poincare inequality~\cite[][Theorem~5.3]{arnold_2018}:
    \begin{equation}\label{eq:cc:discpoincare}
        \|\mub_h\| \leq C_\mathrm{2} \|\curl \mub_h\| = C_2 \|\curl \vb_h^\circ\|.
    \end{equation}
    Using~\eqref{eq:cc:bdryscale}\eqref{eq:cc:hodge}\eqref{eq:cc:discpoincare} and the inverse inequality~\eqref{eq:inv}, we have
    \begin{equation}
        \begin{aligned}
            \|\vb_h - \grad r_h\| &\leq \|\mub_h\| + \|\vb_h^\partial\| \\
            &\leq C_2 \|\curl \vb_h^\circ\| + C_1 h^{1/2} \|\gammat\vb_h\|_{\Lb^2(\Gamma)} \\
            &\leq C_2 (\|\curl \vb_h\| + \|\curl \vb_h^\partial\|) + C_1 h^{1/2} \|\gammat\vb_h\|_{\Lb^2(\Gamma)} \\
            &\leq C_2 \|\curl \vb_h\| + C_2 C_{\text{inv}} h^{-1}\|\vb_h^\partial\| + C_1 h^{1/2} \|\gammat\vb_h\|_{\Lb^2(\Gamma)} \\
            &\leq C_2 \|\curl \vb_h\| + C_2C_{\text{inv}}C_1 h^{-1/2} \|\gammat\vb_h\|_{\Lb^2(\Gamma)} + C_1 h^{1/2} \|\gammat\vb_h\|_{\Lb^2(\Gamma)}\\
            &\leq C_4 |\vb_h|_{\curln, h}.
        \end{aligned}
    \end{equation}
\end{proof}

Now, we have all the ingredients to show the inf-sup estimate:
\begin{lemma}[inf-sup estimate]\label{thm:infsupasym}
Suppose that the isolated patch condition in Section~\ref{sec:patch} holds. There exists a constant $\beta > 0$ independent of $h$ such that
\begin{equation}\label{eq:infsupasym}
    \sup_{\vb_h \in \Sigmab_h}\frac{\a^\mathrm{asym}_{\curln,h}(\ub_h,\vb_h)}{\|\vb_h\|_{\curln,h}} \geq \beta |\ub_h|_{\curln,h} \quad \forall\,\ub_h\in\Sigmab_h.
\end{equation}
\end{lemma}
\begin{proof}
    Fix $\ub_h \in \Sigmab_h$. Let $\varphib_h := \Rsfb_h \Pib_h^\Gamma(\gammat \ub_h) \in \Sigmab_{h,0}$ (see Lemma~\ref{thm:lift}) and $r_h \in P_{h,0}$ be defined as in Lemma~\ref{thm:discpoincareaux}. We set $\vb_h = \ub_h - \grad r_h + \alpha h^{-1}\varphib_h$ where $\alpha > 0$ is a constant to be determined. Noting that $\curl\grad r_h = 0$ and $\gammat \grad r_h = 0$, we have 
    \begin{equation}
        \begin{aligned}
            \a^\mathrm{asym}_{\curln,h}(\ub_h,\vb_h) &= \|\curl \ub_h\|^2 + \alpha h^{-1}(\curl \ub_h, \curl \varphib_h) - \langle \gammax\curl \ub_h, \gammat \ub_h \rangle  \\
            &\quad\quad - \alpha h^{-1} \langle \gammax\curl \ub_h, \gammat\varphib_h \rangle + \langle \gammax \curl \ub_h, \gammat\ub_h \rangle + \alpha h^{-1} \langle \gammax\curl \varphib_h, \gammat \ub_h\rangle\\
            &= \|\curl \ub_h\|^2 + \alpha h^{-1}(\curl \ub_h, \curl \varphib_h) + \alpha h^{-1} \langle \Pib_h^\Gamma(\gammat \ub_h), \gammat \ub_h\rangle\\
            &= \|\curl \ub_h\|^2 + \alpha h^{-1}(\curl \ub_h, \curl \varphib_h) + \alpha h^{-1} \|\Pib_h^\Gamma(\gammat \ub_h)\|^2_{\Lb^2(\Gamma)} \\
            &\geq \frac{1}{2} \|\curl \ub_h\|^2 - \frac{1}{2}\alpha^2h^{-2}\|\curl \varphib_h\|^2 + \alpha h^{-1} \|\Pib_h^\Gamma(\gammat \ub_h)\|^2_{\Lb^2(\Gamma)} \\
            &\geq \frac{1}{2} \|\curl \ub_h\|^2 - \frac{1}{2}\alpha^2C_{\text{inv}}h^{-4}\|\varphib_h\|^2 + \alpha h^{-1} \|\Pib_h^\Gamma(\gammat \ub_h)\|^2_{\Lb^2(\Gamma)} \\
            &\geq \frac{1}{2} \|\curl \ub_h\|^2 + \alpha \left(1 - \frac{1}{2}\alpha C_{\text{inv}}C_{\text{est}}\right)h^{-1}\|\Pib_h^\Gamma(\gammat \ub_h)\|^2_{\Lb^2(\Gamma)}.
        \end{aligned}
    \end{equation}  
    Here, we have used the fact that $\gammat \varphib_h = \zerob$ in the second step; the property~\eqref{eq:bddcontrol} in the third step; the Young's inequality in the fourth step; the inverse inequality~\eqref{eq:inv} in the fifth step; and the estimate~\eqref{eq:bddtestfuncbound} in the last step. A suitable choice of $\alpha$ (e.g., $\alpha = 1/C_{\text{inv}}C_{\text{est}}$) and the norm equivalence in~\eqref{eq:tracenormequiv} yield that $\a^\mathrm{asym}_{\curln,h}(\ub_h,\vb_h) \geq C_1|\ub_h|^2_{\curln,h}$.
    By the inverse inequality~\eqref{eq:inv}, the stability of $\Rsfb_h$~\eqref{eq:bddtestfuncbound} and the estimate~\eqref{eq:discpoincareaux}, we have $\|\vb_h\|_{\curln,h} \leq C_2|\ub_h|_{\curln,h}$ and we can conclude.
\end{proof}

\begin{remark}[Failure of using the \Nedelec elements of the second kind]\label{rmk:type2fail}
    If the discrete space is chosen as the \Nedelec elements \emph{of the second kind}:
    \begin{equation}
        \Sigmab_h^{\mathrm{2nd}} := \{\vb_h \in \Hb(\curln): \vb_h|_T \in \Pcal_k(T)^3, T \in \Tcal_h\}, k\geq 1,
    \end{equation}
    the inf-sup estimate in Lemma~\ref{thm:infsupasym} does not hold. 
    Recall that the key step in establishing~\eqref{eq:infsupasym} is to gain sufficient control over the boundary norm $\|\gammat\ub_h\|_{\Lb^2(\Gamma)}$ via $\langle \gammax \curl \vb_h, \gammat \ub_h  \rangle$. The failure of $\Sigmab_h^\text{2nd}$ is due to a richer trace space $\gammat \Sigmab_h^{\mathrm{2nd}}$ than $\gammat \Sigmab_h$, while $\gammax \curl \Sigmab_h^{\mathrm{2nd}}$ is equal to $\gammax \curl \Sigmab_h$ (see~\cite[][Section~7.5]{arnold_2018}). 
\end{remark}

\subsection{Necessary regularity for consistency}\label{sec:regconsis}
In the subsequent error analysis for the two model problems~\eqref{eq:cc} and~\eqref{eq:mad}, we will rely on \emph{consistency} of the schemes (see~\eqref{eq:cc:consis} and~\eqref{eq:mag:consissd}), which requires plugging exact solutions into the first argument of the bilinear form $\a^\mathrm{asym}_{\curln,h}$. To make $\a^\mathrm{asym}_{\curln,h}(\ub, \vb_h)$ well-defined for $\vb_h \in \Sigmab_h$, we need $\gammat \ub \in \Lb^2(\Gamma)$ and $\gammax \curl \ub \in \Lb^2(\Gamma)$. Recall that the Dirichlet trace operator $\gamma$ maps $H^{s}(\Omega)$ boundedly to $H^{s-1/2}(\Gamma)$ for $s > 1/2$ on bounded Lipschitz domains~\cite[][Theorem~1.5.1.2]{grisvard_2011}. A sufficient regularity is $\ub \in \Hb^{1/2+\epsilon}(\curln), \epsilon > 0$, where
\begin{equation}\label{def:Hscurl}
    \Hb^s(\curln) := \{\vb \in \Hb^s(\Omega): \curl \vb \in \Hb^s(\Omega)\}.
\end{equation}
Therefore, $\ub\in\Hb^{1/2+\epsilon}(\curln)$ is assumed throughout in the subsequent sections. Readers are referred to~\cite[][Section~5.2]{pietro_2012} for a refined technique addressing minimal regularity solutions, which is beyond the scope of this work.

\section{The curl-elliptic problem}\label{sec:curlcurl}
In this section we consider the $\curln$-elliptic problem~\cite{hiptmair_2002}:
    \begin{equation}\label{eq:cc}
        \begin{aligned}
        \curl\curl\ub  &= \fb && \text{in }\Omega,\\
        \div \ub &= 0 && \text{in }\Omega,\\
        \gammat\ub &= \zerob && \text{on }\Gamma,
        \end{aligned}
    \end{equation}
    where $\fb$ is assumed to be solenoidal. A Lagrange multiplier $p \in H^1_0$ is introduced to enforce the divergence constraint. The variational form of~\eqref{eq:cc} seeks $\ub \in \Hzcurl$ and $p \in H^1_0$ such that
    \begin{equation}\label{eq:ccwf}
        \B_{\curln}((\ub, p), (\vb, q)) = (\fb, \vb) \quad \forall\,(\vb, q) \in \Hzcurl\times H^1_0,
    \end{equation}
    where
    \begin{equation}
        \B_{\curln}((\ub, p), (\vb, q)) := \a_{\curln}(\ub, \vb) + (\grad p, \vb) + (\grad q, \ub).
    \end{equation}
    It is easy to see that $p = 0$ as the solution of~\eqref{eq:ccwf}.
    The standard FEEC discretization of~\eqref{eq:cc-intro} seeks $(\ub_h, p_h) \in \Sigmab_{h,0} \times P_{h,0}$ such that
    \begin{equation}\label{eq:cc:std}
        \B_{\curln}((\ub_h, p_h), (\vb_h, q_h)) = (\fb, \vb_h) \quad \forall\,(\vb_h, q_h) \in \Sigmab_{h,0} \times P_{h,0}.
    \end{equation}
    Using the penalty-free asymmetric Nitsche's method, we discretize~\eqref{eq:cc} by seeking $(\ub_h, p_h) \in \Sigmab_{h} \times P_{h,0}$ such that
    \begin{equation}\label{eq:ccasym}
        \B^{\mathrm{asym}}_{\curln}((\ub_h, p_h), (\vb_h, q_h)) = (\fb, \vb_h) \quad \forall\,(\vb_h, q_h) \in \Sigmab_{h} \times P_{h,0},
    \end{equation}
    where
    \begin{equation}
        \begin{aligned}
            \B^{\mathrm{asym}}_{\curln}((\ub_h, p_h), (\vb_h, q_h)) := \a^{\mathrm{asym}}_{\curln, h}(\ub_h, \vb_h) + (\grad p_h, \vb_h) + (\grad q_h, \ub_h).
        \end{aligned}
    \end{equation}
    \subsection{Consistency, boundedness and inf-sup stability}
    It is straightforward to verify that the scheme~\eqref{eq:ccasym} is consistent assuming $\ub \in \Hb^s(\curln)$ for some $s > 1/2$ (see~Section~\ref{sec:regconsis}). That is, if $\ub \in \Hb^s(\curln)$ solves~\eqref{eq:cc}, then
    \begin{equation}\label{eq:cc:consis}
        \B^{\mathrm{asym}}_{\curln}((\ub, 0), (\vb_h, q_h)) = (\fb, \vb_h) \quad \forall\,(\vb_h, q_h) \in \Sigmab_{h} \times P_{h,0}.
    \end{equation}
    Recall the definition of $\|\cdot\|_{\curln,h}$ in~\eqref{def:hseminorm}. We further define norms:
    \begin{equation}
        \begin{aligned}
        \vertiii{\ub}_{\curln, h}^2 &:= \|\ub\|_{\curln,h}^2 + h\|\gammax\curl\ub\|^2_{\Lb^2(\Gamma)} &&\text{ for }\ub \in \Sigmab_h + \Hb^{1/2+\epsilon}(\curln),\\
        \|(\ub, p)\|_{h}^2 &:= \|\ub\|_{\curln, h}^2 + \|\grad p\|^2 &&\text{ for }(\ub, p) \in \left(\Sigmab_h + \Hb^{1/2+\epsilon}(\curln)\right) \times H^1_0,\\
        \vertiii{(\ub, p)}^2_{h} &:= \vertiii{\ub}_{\curln, h}^2 + \|\grad p\|^2 &&\text{ for }(\ub, p) \in \left(\Sigmab_h + \Hb^{1/2+\epsilon}(\curln)\right) \times H^1_0.
        \end{aligned}
    \end{equation}
    By the inverse inequalities~\eqref{eq:inv}\eqref{eq:invtrace}, it is easy to see that $\vertiii{\cdot}_{\curln, h}$ and $\|\cdot\|_{\curln,h}$ are equivalent on $\Sigmab_h$. Well-posedness and stability of~\eqref{eq:ccasym} hold provided the boundedness and inf-sup stability of $\B^{\mathrm{asym}}_{\curln}$ with respect to the norm $\vertiii{\cdot}_h$. The former is straightforward by the Cauchy-Schwarz inequality, while the latter relies on the inf-sup stability of $\a^{\mathrm{asym}}_{\curln,h}$ established in Lemma~\ref{thm:infsupasym}.
    
    \begin{lemma}[boundedness]\label{thm:ccboundedness}
        There exists a constant $C_{\curln} > 0$ independent of $h$ such that
        \begin{equation}\label{eq:cc:bddness}
            |\B^{\mathrm{asym}}_{\curln}((\ub, p), (\vb, q))| \leq C_{\curln} \vertiii{(\ub, p)}_{h} \vertiii{(\vb, q)}_{h} \quad \forall\,(\ub, p),(\vb, q) \in \left(\Sigmab_h + \Hb^{1/2+\epsilon}(\curln)\right) \times H^1_0.
        \end{equation}
    \end{lemma}
    \begin{lemma}[inf-sup estimate]\label{thm:ccinfsup}
        Suppose that the isolated patch condition in Section~\ref{sec:patch} holds. There exists a constant $\beta_{\curln} > 0$ independent of $h$ such that
        \begin{equation}\label{eq:cc:infsup}
             \sup_{(\vb_h, q_h) \in \Sigmab_h \times P_{h,0}} \frac{\B^{\mathrm{asym}}_{\curln}((\ub_h, p_h), (\vb_h, q_h))}{\vertiii{(\vb_h, q_h)}_{h}} \geq \beta_{\curln} \vertiii{(\ub_h, p_h)}_{h}.
        \end{equation}
    \end{lemma}
    \begin{proof}
        We need to pick candidate test functions given $(\ub_h, p_h) \in \Sigmab_h \times P_{h,0}$. By Lemma~\eqref{thm:infsupasym}, there exists $\widetilde{\ub}_h \in \Sigmab_h$ such that $\a^{\mathrm{asym}}_{\curln,h}(\ub_h, \widetilde{\ub}_h) \geq C_1 |\ub_h|^2_{\curln,h}$ and $\|\widetilde{\ub}_h\|_{\curln,h} \leq C_2 |\ub_h|_{\curln,h}$. By Lemma~\ref{thm:discpoincareaux}, there exists $r_h \in P_{h,0}$ such that $\|\ub_h - \grad r_h\| \leq C_3 |\ub_h|_{\curln,h}$.
         Perform the following calculation:
        \begin{equation}\label{eq:cc:calc}
            \begin{aligned}
                \B^{\mathrm{asym}}_{\curln}((\ub_h, p_h), &(\widetilde{\ub}_h,0)) &&\geq C_1 |\ub_h|^2_{\curln,h} + (\grad p_h, \widetilde{\ub}_h) \\
                & &&\geq C_1 |\ub_h|^2_{\curln,h} - \frac{C_2^2}{2C_1}\|\grad p_h\|^2 - \frac{C_1}{2C^2_2} \|\widetilde{\ub}_h\|^2 \\
                & &&\geq \frac{C_1}{2}|\ub_h|^2_{\curln,h} - \frac{C_2^2}{2C_1}\|\grad p_h\|^2, \\
                \B^{\mathrm{asym}}_{\curln}((\ub_h, p_h), &(\grad p_h, 0)) &&= \a^\mathrm{asym}_{\curln,h}(\ub_h, \grad p_h) + \|\grad p_h\|^2  = \|\grad p_h\|^2 \\
                \B^{\mathrm{asym}}_{\curln}((\ub_h, p_h), &(\zerob, r_h)) &&= (\grad r_h, \ub_h)  \\
                & &&= (\ub_h, \ub_h) + (\grad r_h - \ub_h, \ub_h)  \\
                & &&\geq \frac{1}{2}\|\ub_h\|^2 - \frac{1}{2}\|\grad r_h - \ub_h\|^2 \\
                & &&\geq \frac{1}{2}\|\ub_h\|^2 - \frac{C_3^2}{2} |\ub_h|^2_{\curln,h}.
            \end{aligned}
        \end{equation}
        Choose $\vb_h = \widetilde{\ub}_h + \delta_1 \grad p_h$ and $q_h = \delta_2 r_h$ where $\delta_1, \delta_2 > 0$ are constants to be determined. Combining the estimates in~\eqref{eq:cc:calc}, we obtain
        \begin{equation}
            \begin{aligned}
                \B^{\mathrm{asym}}_{\curln}((\ub_h, p_h), (\vb_h, q_h)) 
                &\geq \left(\frac{C_1}{2} - \frac{C_3^2}{2}\delta_2\right)|\ub_h|^2_{\curln,h} + \left(\delta_1 - \frac{C_2^2}{2C_1}\right)\|\grad p_h\|^2 + \frac{\delta_2}{2}\|\ub_h\|^2.
            \end{aligned}
        \end{equation}
        Choose $\delta_1 > \frac{C_2^2}{2C_1}$ and $\delta_2 < \frac{C_1}{C_3^2}$ and we obtain
        \begin{equation}
            \B^{\mathrm{asym}}_{\curln}((\ub_h, p_h), (\vb_h, q_h)) 
                \geq C_4 \|(\ub_h, p_h)\|^2_{h}.
        \end{equation}
        It remains to show that $\|(\vb_h, q_h)\|_{h} \leq C_5 \|(\ub_h, p_h)\|_{h}$, which is straightforward by the triangle inequality and the above estimates. The conclusion follows from the norm equivalence of the $\vertiii{\cdot}_{\curln,h}$ and $\|\cdot\|_{\curln,h}$ norms.
    \end{proof}
    \subsection{A priori error estimates}\label{sec:cc:err}
    The following quasi-optimal error estimate of the discretization~\eqref{eq:ccasym} in the $\vertiii{\cdot}_{\curln,h}$-norm follows immediately from the inf-sup stability~\eqref{eq:cc:infsup} and the boundedness~\eqref{eq:cc:bddness}. 
    \begin{theorem}[quasi-optimality in $\vertiii{\cdot}_{\curln,h}$]\label{thm:cc:quasiopt}
        Suppose that the isolated patch condition in Section~\ref{sec:patch} holds. Let $(\ub, p) \in \Hzcurl\times H^1_0$ and $(\ub_h, p_h) \in \Sigmab_h \times P_{h,0}$ solve~\eqref{eq:cc} and~\eqref{eq:ccasym}, respectively. Suppose $\ub \in \Hb^{1/2+\epsilon}(\curln)$. There exists a constant $C > 0$ independent of $h$ such that
        \begin{equation}\label{eq:cc:quasiopt}
            \vertiii{\ub - \ub_h}_{\curln, h}\leq  C\inf_{\ub_I \in \Sigmab_h}\vertiii{\ub - \ub_I}_{\curln, h}.
        \end{equation}
    \end{theorem}
    \begin{proof}
        First notice that $p = p_h = 0$.
        Take arbitrary $(\ub_I, p_I) \in \Sigmab_h \times P_{h,0}$. By the inf-sup stability~\eqref{eq:cc:infsup}, the consistency~\eqref{eq:cc:consis}, and the boundedness~\eqref{eq:cc:bddness}, we have
        \begin{equation}
            \begin{aligned}
                \beta_{\curln}\vertiii{(\ub_h - \ub_I, p_h - p_I)}_{h} &\leq \sup_{(\vb_h, q_h) \in \Sigmab_h \times P_{h,0}} \frac{\B^{\mathrm{asym}}_{\curln}((\ub_h - \ub_I, p_h - p_I), (\vb_h, q_h))}{\|(\vb_h, q_h)\|_{h}}\\
                &\leq  \sup_{(\vb_h, q_h) \in \Sigmab_h \times P_{h,0}} \frac{\B^{\mathrm{asym}}_{\curln}((\ub - \ub_I, p - p_I), (\vb_h, q_h))}{\|(\vb_h, q_h)\|_{h}} \\
                &\leq  C_{\curln}\vertiii{(\ub - \ub_I, p - p_I)}_{h}.
            \end{aligned}
        \end{equation}
        Applying the triangle inequality and taking $p_I = 0$, we can conclude.
    \end{proof}
    
    The next lemma states the interpolation error in the $\vertiii{\cdot}_{\curln,h}$-norm using $\INDav$ introduced in Section~\ref{sec:IND}. Due to the presence of $\Lb^2$ boundary term in the $\vertiii{\cdot}_{\curln,h}$-norm, the approximation rates differ in the case of homogeneous and non-homogeneous tangential BCs. 
    \begin{lemma}[interpolation errors]\label{thm:ccapprox}
        Suppose $\ub \in \Hb^{k+1}(\curln)$. There exists a constant $C > 0$ independent of $h$ such that
        \begin{equation}\label{eq:cc:approxnonhomo}
            \vertiii{\ub - \INDav \ub}_{\curln, h} \leq Ch^{k}|\ub|_{\Hb^{k+1}} + Ch^{k+1}|\curl\ub|_{\Hb^{k+1}}.
        \end{equation}
        If further $\gammat \ub = \zerob$, then
        \begin{equation}\label{eq:cc:approxhom}
            \vertiii{\ub - \INDav \ub}_{\curln, h} \leq Ch^{k+1}\left(|\ub|_{\Hb^{k+1}} + |\curl\ub|_{\Hb^{k+1}}\right).
        \end{equation}
    \end{lemma}
    \begin{proof}
        Note that $\INDav$ preserves homogeneous tangential BCs. The results are direct consequences of~\eqref{eq:approxIND} and~\eqref{eq:traceapproxIND}. 
    \end{proof}
    \begin{corollary}[a priori error estimate]\label{thm:cc:err}
        Suppose that the isolated patch condition in Section~\ref{sec:patch} holds. Let $(\ub, p) \in \Hzcurl\times H^1_0$ and $(\ub_h, p_h) \in \Sigmab_h \times P_{h,0}$ solve~\eqref{eq:cc} and~\eqref{eq:ccasym}, respectively. Assume $\ub \in \Hb^{k+1}(\curln)$. There exists a constant $C > 0$ independent of $h$ such that
        \begin{equation}
            \vertiii{\ub - \ub_h}_{\curln, h}\leq Ch^{k+1}\left(|\ub|_{\Hb^{k+1}} + |\curl\ub|_{\Hb^{k+1}}\right).
        \end{equation}
    \end{corollary}
    \begin{proof}
        The estimate follows from~\eqref{eq:cc:quasiopt} and~\eqref{eq:cc:approxhom}.
    \end{proof}
    \subsection{Non-homogeneous tangential BCs}\label{sec:cc:nonhomo}
    For the sake of presentation, we have focused on the case of homogeneous tangential BCs $\gammat \ub = \zerob$ in the above analysis. In this section, we briefly discuss the subtlety in the case of non-homogeneous tangential BCs $\gammat \ub = \gammat \gb $ where $\gb \in \Hcurl$.

    To accommodate $\gammat\gb \neq \zerob$, the most straightforward approach is to seek $(\ub_h, p_h) \in \Sigmab_h \times P_{h,0}$ such that 
    \begin{equation}\label{eq:ccasymnonhomo}
        \B^{\mathrm{asym}}_{\curln}((\ub_h, p_h), (\vb_h, q_h)) = (\fb, \vb_h) + \langle \gammax \curl \vb_h, \gammat\gb\rangle \quad \forall\,(\vb_h, q_h) \in \Sigmab_{h} \times P_{h,0},
    \end{equation}
    which is consistent with the exact solution. The quasi-optimality as stated in Theorem~\ref{thm:cc:quasiopt} still holds, but the discretization error in the $\vertiii{\cdot}_{\curln,h}$-norm is only of order $h^k$ for a smooth $\ub$ in view of~\eqref{eq:cc:approxnonhomo}. This implies suboptimal convergence in $\Hcurl$-norm, which is confirmed by numerical experiments (see Figure~\ref{fig:cc:convergence}). 
    
    This suboptimality is due to the incomplete/trimmed polynomial spaces of $\Sigmab_h$, the \Nedelec elements of the first kind. Similar issue has been discussed in the context of $\Hcurl$-interface problems using Nitsche's method~\cite{casagrande_2016}.

    A remedy to restore optimal convergence in $\Hcurl$-norm is to replace the original boundary data $\gammat\gb$ with interpolated boundary data $\gammat \INDav \gb$ (see definition in Section~\ref{sec:IND}). That is, seek $(\ub_h, p_h) \in \Sigmab_h \times P_{h,0}$ such that
    \begin{equation}\label{eq:ccasymnonhomo2}
        \B^{\mathrm{asym}}_{\curln}((\ub_h, p_h), (\vb_h, q_h)) = (\fb, \vb_h) + \langle \gammax \curl \vb_h, \gammat \INDav \gb\rangle \quad \forall\,(\vb_h, q_h) \in \Sigmab_{h} \times P_{h,0}.
    \end{equation}
    In contrast to~\eqref{eq:ccasymnonhomo}, this scheme is slightly inconsistent, but the inconsistency error can be controlled well. By solving~\eqref{eq:ccasymnonhomo2}, optimal convergence in $\Hcurl$-norm can be achieved. The proof is straightforward based on the previous analysis and thus omitted for brevity. Additionally, we point out that this trick is in some sense equivalent to first reformulating the continuous problem into a form with homogeneous tangential BCs and then applying the scheme~\eqref{eq:ccasym}.

\section{The magnetic advection-diffusion problem}\label{sec:mad}
In this section, we consider the \emph{magnetic advection-diffusion} boundary value problem~\cite{heumann_2013,li_2026}:
\begin{equation}\label{eq:mad}
        \begin{aligned}
        \eps \curl\curl\ub + \L_\betab \ub + \alpha \ub &= \fb && \text{in }\Omega,\\
        \gammat\ub &= \zerob && \text{on }\Gamma\setminus\Gamma^-, \\
        \ub &= \zerob && \text{on }\Gamma^-,
        \end{aligned}
    \end{equation}
where 
\begin{itemize}
    \item $\L_\betab \ub : = \curl\ub \times \betab + \grad (\ub\cdot\betab)$ is the Lie derivative (for differential 1-form);
    \item $\Gamma^- := \{x\in\Gamma: \betab(x)\cdot\nb(x) < 0\}$ is the inflow boundary;
    \item $\betab \in \Wb^{1,\infty}$ is a given advection field;
    \item $\alpha \in L^\infty$ is a given reaction coefficient;
    \item $\eps > 0$ is a given diffusion coefficient.
\end{itemize}
Equations of the form~\eqref{eq:mad} arise in magnetohydrodynamics (MHD) and govern the evolution of magnetic field in a conducting fluid (see, e.g.~\cite{spruit_2016}).
In physical scenarios, $\eps$ represents the magnetic diffusivity and is often very small. This renders the problem singularly perturbed and the solutions may exhibit boundary layers. Formally, the limit problem as $\eps \rightarrow 0$ is the pure magnetic advection boundary value problem:
\begin{equation}\label{eq:magadvec}
    \begin{aligned}
    \L_\betab \ub + \alpha \ub &= \fb && \text{in }\Omega,\\
    \ub &= \zerob && \text{on }\Gamma^-.
    \end{aligned}
\end{equation}

The major difficulty in discretization of~\eqref{eq:mad} is numerical instability in the advection-dominated regime, that is, when $\eps \ll |\betab|$. This issue has been extensively studied in the context of the \emph{scalar advection-diffusion equation}:
\begin{equation}\label{eq:scalardiff}
    \begin{aligned}
    -\eps \Delta u + \betab \cdot \grad u + \alpha u &= f && \text{in }\Omega,\\
    u &= 0 && \text{on }\Gamma.
    \end{aligned}
\end{equation}
where a large variety of stabilized methods have been proposed; see~\cite{roos_2008} for a comprehensive review. In the language of \emph{exterior calculus}, both~\eqref{eq:mad} and~\eqref{eq:scalardiff} are instances of a linear advection-diffusion problem for differential $0$-form and $1$-form, respectively~\cite{heumann_2016}. Taking a cue from the stabilization techniques for~\eqref{eq:scalardiff}, many methods have been proposed for~\eqref{eq:mad}, including DG methods~\cite{heumann_2013,wang_2024}, streamline upwind Petrov-Galerkin (SUPG) methods~\cite{li_2026}, exponential fitting methods~\cite{wang_2023}, local projection stabilization (LPS) methods~\cite{luo_2025}, etc.

Well-posedness of~\eqref{eq:mad} and~\eqref{eq:magadvec} is guaranteed by the following assumption~\cite[][Assumption~1.3]{heumann_2013}:
\begin{equation}\label{ass:mad}
    \lambda_{\min}\{(2\alpha - \div \betab)\Ib + \Dsf \betab + (\Dsf \betab)^\top\} \geq \alpha_0 > 0 \text{ a.e. in }\Omega
\end{equation}
where $\Dsf \betab$ is the Jacobian matrix of $\betab$ and $\lambda_{\min}$ is the smallest eigenvalue of the corresponding matrix. We assume~\eqref{ass:mad} throughout this section.

In what follows, we analyze a discretization of~\eqref{eq:mad} combining the penalty-free asymmetric Nitsche's method~\eqref{eq:blasym} for the diffusion term $\curl\curl\ub$ and the SUPG stabilization proposed in~\cite{li_2026} for the advection term $\L_\betab \ub$. Due to the weak imposition of the tangential BCs, the redundant conditions in the $\eps\rightarrow 0$ limit on $\Gamma \setminus \Gamma^-$ (see~\eqref{eq:mad} and~\eqref{eq:magadvec}) are discarded naturally at the discrete level. Benefits of weak imposition of Dirichlet BCs in flow problems have been widely discussed~\cite{burman_2006b,bazilevs_2007,schieweck_2008}. 

\subsection{Discretization}
We first introduce some notations. Denote the discrete operator $\Asf_h: \Sigmab_h \rightarrow \Lb^2$ piecewise as
\begin{equation}
    \Asf_h \ub := \eps \curl \curl \ub + \L_\betab \ub + \alpha \ub \text{ on each } T \in \mathcal{T}_h,
\end{equation}
and the discrete Lie-derivative $\L_{\betab,h}$ as
\begin{equation}
    \L_{\betab,h} \ub := \curl \ub \times \betab + \grad (\ub\cdot\betab) \text{ on each } T \in \mathcal{T}_h.
\end{equation}
Recall the definition of sets of mesh entities~\eqref{def:meshnotations}. Additionally, we assume that each boundary face $F \in \Fcal_h^\Gamma$ lies entirely on $\Gamma^-$ or $\Gamma\setminus\Gamma^-$. Denote the set of faces on the inflow boundary $\Gamma^-$ as
\begin{equation}
    \Fcal^-_h := \{F \in \Fcal_h^\Gamma: F \subset \Gamma^- \}.
\end{equation}
For each interior face $F \in \Fcal_h^\circ$, assign a fixed unit normal vector $\nb_F$ and let $T^+_F$ and $T^-_F$ be the two elements sharing $F$ such that $\nb_F$ points from $T^+_F$ to $T^-_F$. For each boundary face $F \in \Fcal_h^\Gamma$, assign a fixed unit normal vector $\nb_F$ pointing outward of $\Omega$. For a piecewise smooth function $\vb$, let $\vb^\pm_F$ be its full trace on $F$ from $T^\pm_F$. Next, define the jump and average operators as 
\begin{equation}
    \begin{aligned}
        \jump{\vb}_F &:= \vb^+_F - \vb^-_F, \quad &&\avg{\vb}_F := \frac{1}{2}(\vb^+_F + \vb^-_F) &&\text{ for } F \in \Fcal_h^\circ,\\
        \jump{\vb}_F &:= \vb, \quad &&\avg{\vb}_F := \vb &&\text{ for } F \in \Fcal_h^\Gamma.
    \end{aligned}
\end{equation}

Recall that $\a_{\curln}(\ub,\vb) := (\curl\ub,\curl\vb)$ and introduce the following SUPG-stabilized bilinear form for the advection-diffusion term $\L_\betab \ub + \alpha \ub$:
\begin{equation}\label{eq:mag:stabterm}
    \begin{aligned}
    \a_{\mathrm{adv},h}^{\mathrm{SUPG}}(\ub_h, \vb_h) &:= (\L_{\betab,h} \ub_h, \vb_h) + (\alpha \ub_h, \vb_h) - \sum_{F \in \Fcal^\circ_h \cup \Fcal^-_h} \langle \betab\cdot \nb_F \jump{\ub_h}, \avg{\vb_h}\rangle_F \\
        &\quad\quad + \underbrace{\sum_{F \in \Fcal^\circ_h} \left\langle \frac{1}{2}|\betab \cdot \nb_F| \jump{\ub_h}, \jump{\vb_h}\right\rangle_F}_{S^1} + \underbrace{\sum_{T\in\Tcal_h}\delta_T \left(\Asf_h \ub_h, \L_{\betab,h} \vb_h \right)_T}_{S^2}.
    \end{aligned}
\end{equation}

The SUPG scheme~\cite[][Eq.~3.10]{li_2026} for~\eqref{eq:mad} with \emph{tangential} BCs strongly imposed on $\Gamma$ seeks $\ub_h \in \Sigmab_{h,0}$ such that
\begin{equation}\label{eq:supgschemestrongbc}
    \Bsf_{\mathrm{mad},h}(\ub_h, \vb_h) := \eps \a_{\curln}(\ub_h, \vb_h) + \a_{\mathrm{adv},h}^{\mathrm{SUPG}}(\ub_h, \vb_h) = \lsf(\fb, \vb_h) \quad \forall\,\vb_h \in \Sigmab_{h,0},
\end{equation}
where
\begin{equation}
    \lsf(\fb, \vb_h) := (\fb, \vb_h) + \sum_{T\in\Tcal_h} \delta_T (\fb, \L_{\betab,h} \vb_h)_T.
\end{equation}
Here, $\delta_{T} > 0, T \in \Tcal_h$ is a stabilization parameter per element to be determined.

We consider the following scheme: seek $\ub_h \in \Sigmab_h$ such that
\begin{equation}\label{eq:supgscheme}
    \Bsf_{\mathrm{mad},h}^{\mathrm{asym}}(\ub_h, \vb_h):= \eps \a^{\mathrm{asym}}_{\curln,h}(\ub_h, \vb_h) + \a_{\mathrm{adv},h}^{\mathrm{SUPG}}(\ub_h, \vb_h) = \lsf(\fb, \vb_h) \quad \forall\,\vb_h \in \Sigmab_h.
\end{equation}
By comparing the advection-diffusion equation~\eqref{eq:mad} and its $\eps\rightarrow 0$ limit~\eqref{eq:magadvec}, one sees that the tangential BCs on the outflow boundary $\Gamma \setminus \Gamma^-$ are dropped in the limit. This implies that a sharp boundary layer of $\gammat \ub$ may develop near $\Gamma \setminus \Gamma^-$ (see Figure~\ref{fig:mad:boundarylayer}). When $\eps = 0$, the scheme~\eqref{eq:supgscheme} does not impose any tangential BCs on $\Gamma \setminus \Gamma^-$, which mimics the limiting behavior of the continuous level as $\eps \rightarrow 0$. In the numerical experiments in Section~\ref{sec:num:mad}, it indeed leads to better performance in the advection-dominated regime ($\eps \ll h$) than the scheme~\eqref{eq:supgschemestrongbc} with strongly imposed tangential BCs.

\begin{remark}[non-homogeneous BCs]\label{rmk:mad:nonhomo}
    To accommodate non-homogeneous tangential BCs $\gammat \ub = \gammat \gb$ on $\Gamma\setminus\Gamma^-$ and $ \ub = \gb$ on $\Gamma^-$, the right-hand side of~\eqref{eq:supgscheme} should be modified:
    \begin{equation}\label{eq:mad:nonhomo}
        \Bsf_{\mathrm{mad},h}^{\mathrm{asym}}(\ub_h, \vb_h) = \lsf(\fb, \vb_h) + \eps\langle \gammax \curl \vb_h, \gammat\gb\rangle - \sum_{F \in \Fcal^-_h} \langle \betab\cdot \nb_F\, \gb, \vb_h\rangle_F.
    \end{equation}
\end{remark}

\begin{remark}[comparison with the original scheme of~\cite{li_2026}]
     For the sake of simplicity, we have made a slight modification to the original SUPG scheme in~\cite{li_2026}. Originally, the discrete operator $\Asf_h$ accommodates the $S^1$ term while the one here does not. This makes the current SUPG term $S^2$ slightly weaker than the original one, but still sufficient for stability and convergence\footnote{We are grateful to Jindong Wang for the helpful discussions on this aspect.}. Additionally, the \Nedelec elements of \emph{the second kind} are used in~\cite{li_2026} while \emph{the first kind} is used in this work; see discussion in Remark~\ref{rmk:type2fail}.
\end{remark}

\subsubsection{Consistency and coercivity}
Additional regularity of the exact solution $\ub$ of~\eqref{eq:mad} is required for consistency of the scheme~\eqref{eq:supgscheme}. In particular, $\Bsf_{\mathrm{mad},h}^{\mathrm{asym}}(\ub, \vb_h)$ must be well-defined for $\vb_h \in \Sigmab_h$. A sufficient regularity condition is $\ub\in\Hb^1(\curln)$.
It is straightforward to verify that the scheme~\eqref{eq:supgscheme} is consistent with the original problem~\eqref{eq:mad}. Indeed, if $\ub\in \Hb^1(\curln)$ solves~\eqref{eq:mad}, then
\begin{equation}\label{eq:mag:consissd}
    \begin{aligned}
        \Bsf_{\mathrm{mad},h}^{\mathrm{asym}}(\ub, \vb_h) = \lsf(\fb, \vb_h) \quad \forall\,\vb_h \in \Sigmab_h.
    \end{aligned}
\end{equation}

Define the energy norm
\begin{equation}\label{eq:enormsupg}
        \begin{aligned}
            \|\ub_h\|^2_{\mathrm{mad},h} &:= \eps\|\curl \ub_h\|^2 + \|\ub_h\|^2 + \sum_{F\in \Fcal_h} \left\||\betab\cdot\nb_F|^{1/2}\jump{\ub_h}\right\|^2_F + \sum_{T \in \Tcal_h}\delta_T \|\L_{\betab,h} \ub_h\|_{\Lb^2(T)}^2.
        \end{aligned}
\end{equation}
\begin{lemma}[discrete coercivity]\label{thm:coersupg}
    Suppose the stabilization parameter $\delta_T$ satisfies
    \begin{equation}
        \delta_T \leq \min\left\{\frac{h_T^2}{2C^2_{\mathrm{inv}}\eps}, \frac{\alpha_0}{2\|\alpha\|_{L^\infty(T)}^2}\right\} \quad \forall\,T \in \Tcal_h,
    \end{equation}
    where $C_{\mathrm{inv}}$ is the constant in the inverse inequality~\eqref{eq:inv} and $\alpha_0$ is the constant in the condition~\eqref{ass:mad}.
    Then, there exists a constant $\beta_{\mathrm{mad}} > 0$ such that
    \begin{equation}
        \Bsf_{\mathrm{mad},h}^{\mathrm{asym}}(\ub_h, \ub_h) \geq \beta_{\mathrm{mad}}\|\ub_h\|^2_{\mathrm{mad},h} \quad \forall\,\ub_h \in \Sigmab_h.
    \end{equation}
\end{lemma}
\begin{proof}
    The proof is essentially a repetition of the argument in~\cite[][Lemma~5]{li_2026} and thus omitted.
\end{proof}

An additional technicality arises from adopting the penalty-free asymmetric bilinear form $\a^{\mathrm{asym}}_{\curln,h}$. On one hand, an inf-sup stability of $\Bsf_{\mathrm{mad},h}^{\mathrm{asym}}$ that involves the boundary $\Lb^2$-control $\|\gammat\ub_h\|_{\Lb^2(\Gamma)}$ is unclear due to the interference of the advection term. On the other hand, $\Bsf_{\mathrm{mad},h}^{\mathrm{asym}}$ is only coercive with respect to an energy norm~\eqref{eq:enormsupg}, which does not include the boundary $\Lb^2$ norm. 
A special interpolant is constructed to handle this difficulty. This is the goal of the next section.

\subsubsection{A special interpolant}
In this section, we construct an interpolation operator $\boldsymbol{\mathcal{I}}^*_h: \Hb^{1/2+\epsilon}(\curln) \rightarrow \Sigmab_h$ such that 
\begin{itemize}
    \item[1.] \emph{approximation}:
    \begin{equation}\label{eq:mad:approxistar}
        \begin{aligned}
            |\vb - \boldsymbol{\mathcal{I}}^*_h\vb|_{\Hb^m} &\leq  Ch^{r-m} |\vb|_{\Hb^{r}},\\
            |\curl(\vb - \boldsymbol{\mathcal{I}}^*_h\vb)|_{\Hb^m} &\leq  Ch^{r-m} |\curl \vb|_{\Hb^{r}}, 
        \end{aligned}
    \end{equation}
    for $m \in [0,r], r \in [0,k+1]$;
    \item[2.] \emph{boundary orthogonality}:
    \begin{equation}\label{eq:mad:orthistar}
        \left\langle \gammax\curl ( \vb - \boldsymbol{\mathcal{I}}^*_h\vb), \taub^\Gamma_h\right\rangle = 0 \quad \forall\,\taub^\Gamma_h \in \Mb^\Gamma_h.
    \end{equation}
\end{itemize}
Such conditions and the following construction are motivated by~\cite[][Section~7.2]{burman_2012} which addresses scalar advection-diffusion problems with nodal elements. By~\eqref{eq:mad:approxistar} and the trace inequality~\eqref{eq:multtraceineq}, the same interpolation estimates as~\eqref{eq:traceapproxIND} can be deduced for $\boldsymbol{\mathcal{I}}^*_h$ as well.

Recall the interpolation operator $\INDav: \Lb^2 \rightarrow \Sigmab_h$ in~\eqref{sec:IND}, the boundary projection $\Pib^\Gamma_h: \Lb^2(\Gamma) \rightarrow \Mb^\Gamma_h$ in~\eqref{def:PiMb} and the lifting operator $\Rsfb_h:\Mb^\Gamma_h \rightarrow \Sigmab_{h,0}$ in Lemma~\ref{thm:lift}. Thereby, we introduce interpolation operator $\boldsymbol{\mathcal{I}}^*_h$ as:
    \begin{equation}\label{def:ui}
    \boldsymbol{\mathcal{I}}^*_h \vb := \INDav \vb + \Rsfb_h\Pib^\Gamma_h\gammax\curl(\vb - \INDav \vb).
    \end{equation}
    The property~\eqref{eq:bddcontrol} of $\Rsfb_h$ implies~\eqref{eq:mad:orthistar}.
    By the inverse inequality~\eqref{eq:inv}, the stability estimate~\eqref{eq:bddtestfuncbound}, the $\Lb^2$-stability of $\Pib^\Gamma_h$ and the approximation property~\eqref{eq:traceapproxIND}, we deduce
    \begin{equation}
        \begin{aligned}
            |\Rsfb_h\Pib^\Gamma_h\gammax\curl(\vb - \INDav \vb)|_{\Hb^m} &\leq C h^{-m} \|\Rsfb_h\Pib^\Gamma_h\gammax\curl(\vb - \INDav \vb)\| \\
             &\leq C h^{3/2 - m} \|\Pib^\Gamma_h \gammax \curl(\vb - \INDav \vb)\|_{\Lb^2(\Gamma)} \\
            &\leq C h^{3/2 - m} \|\gammax \curl(\vb - \INDav \vb)\|_{\Lb^2(\Gamma)} \\
            &\leq C h^{r - m + 1} |\curl \vb|_{\Hb^r}
        \end{aligned}
    \end{equation}
    for $m\in[0,r],\; r\in[0,k+1]$. This verifies~\eqref{eq:mad:approxistar}. We additionally point out that $\boldsymbol{\mathcal{I}}^*_h$ preserves homogeneous tangential BCs by construction.
\subsection{A priori error estimates}\label{sec:mad:err}
We are now ready to derive a priori error estimates of the scheme~\eqref{eq:supgscheme} for a smooth solution. For simplicity, we use a unified stabilization parameter $\delta$ such that $\delta_T = \delta$ for all $T \in \mathcal{T}_h$.
\begin{theorem}[a priori error estimate]\label{thm:errsupg}
    Suppose that the isolated patch condition in Section~\ref{sec:patch} and the assumption of Lemma~\ref{thm:coersupg} hold. Let $\ub \in \Hb^{k+1}(\curln)$ and $\ub_h \in \Sigmab_h$ solve~\eqref{eq:mad} and~\eqref{eq:supgscheme}, respectively. There exists a constant $C > 0$ independent of $h$ such that
    \begin{equation}\label{eq:mad:err}
        \|\ub - \ub_h\|_{\mathrm{mad},h} \leq C\left(\delta^{-1/2} h + h^{1/2} + \delta^{1/2}\right) h^{k}|\ub|_{\Hb^{k+1}} + C\left(\eps^{1/2}h + \eps\delta^{1/2}\right) h^{k}|\curl\ub|_{\Hb^{k+1}}.
    \end{equation}
\end{theorem}
\begin{proof}
    Denote $\ub_I := \boldsymbol{\mathcal{I}}^*_h \ub$ and $\etab_h := \ub_I - \ub_h$. By coercivity of $\Bsf_{\mathrm{mad},h}^{\mathrm{asym}}$ (see Lemma~\ref{thm:coersupg}) and the consistency~\eqref{eq:mag:consissd}, we have
    \begin{equation}\label{eq:mad:tmpest}
        \begin{aligned}
            \beta_{\mathrm{mad}}\|\etab_h\|^2_{\mathrm{mad},h} &\leq \Bsf_{\mathrm{mad},h}^{\mathrm{asym}}(\etab_h, \etab_h) =  \Bsf_{\mathrm{mad},h}^{\mathrm{asym}}(\ub_I - \ub, \etab_h).
        \end{aligned}
    \end{equation}
    The right-hand side is written as a summation of seven terms:
    \begin{equation}
        \Bsf_{\mathrm{mad},h}^{\mathrm{asym}}(\ub_I - \ub, \etab_h) := \sum_{i=0}^6 I_i,
    \end{equation}
    where
    \begin{equation}
        \begin{aligned}
            I_0 &:= \underbrace{- \eps \langle \gammax \curl (\ub_I - \ub), \gammat \etab_h\rangle}_{I_{0,1}} + \underbrace{\eps \langle \gammat (\ub_I - \ub), \gammax \curl \etab_h\rangle}_{I_{0,2}}, \\
            I_1 &:= \eps (\curl (\ub_I - \ub), \curl \etab_h), \\
            I_2 &:= (\L_{\betab,h} (\ub_I - \ub), \etab_h) - \sum_{F \in \Fcal^\circ_h \cup \Fcal^-_h} \langle \betab\cdot \nb_F \jump{\ub_I - \ub}, \avg{\etab_h}\rangle_F, \\
            I_3 &:= (\alpha (\ub_I - \ub), \etab_h), \\
            I_4 &:= \sum_{T \in \mathcal{T}_h} \delta_T (\eps \curl\curl (\ub_I - \ub), \L_{\betab,h} \etab_h)_T, \\
            I_5 &:= \sum_{T \in \mathcal{T}_h} \delta_T (\L_{\betab,h} (\ub_I - \ub), \L_{\betab,h} \etab_h)_T, \\
            I_6 &:= \sum_{T \in \mathcal{T}_h} \delta_T (\alpha (\ub_I - \ub), \L_{\betab,h} \etab_h)_T.
        \end{aligned}
    \end{equation}
    The terms $I_2, I_3, I_5, I_6$ are estimated using the approximation property~\eqref{eq:mad:approxistar} in the same way as~\cite[][Theorem~1]{li_2026}. We thus omit the lengthy but elementary derivations to avoid redundancy and only present the resultant estimates:
    \begin{equation}
        |I_2| + |I_3| + |I_5| + |I_6| \leq C\left(\delta^{-1/2} h + h^{1/2} + \delta^{1/2}\right) h^{k}|\ub|_{\Hb^{k+1}} \|\etab_h\|_{\mathrm{mad},h}.
    \end{equation}
    
    The term $I_0$ is new due to usage of $\a^{\mathrm{asym}}_{\curln,h}$. To estimate $I_{0,1}$, we rely on the orthogonality~\eqref{eq:mad:orthistar}, the trace \Poincare inequality~\eqref{eq:tracepoincare}, the trace inequality~\eqref{eq:multtraceineq} and the approximation property~\eqref{eq:mad:approxistar} to deduce
    \begin{equation}
        \begin{aligned}
            |I_{0,1}| &= |\eps \langle \gammax \curl (\ub_I - \ub), \gammat \etab_h - \Pib^\Gamma_h \gammat \etab_h\rangle|\\
            &\leq C \eps \|\gammax \curl (\ub_I - \ub)\|_{\Lb^2(\Gamma)} \|\gammat \etab_h - \Pib^\Gamma_h \gammat \etab_h\|_{\Lb^2(\Gamma)} \\
            &\leq C \eps \|\gammax \curl (\ub_I - \ub)\|_{\Lb^2(\Gamma)} h^{1/2} \|\curl \etab_h\| \\
            &\leq C \eps^{1/2} h^{k+1} |\curl \ub|_{H^{k+1}} \|\etab_h\|_{\mathrm{mad},h}.
        \end{aligned}
    \end{equation}
    The second term $I_{0,2} = 0$ since $\gammat \ub = \zerob$ and $\boldsymbol{\mathcal{I}}^*_h$ preserves homogeneous tangential BCs. Using the estimate~\eqref{eq:mad:approxistar}, we have 
    \begin{equation}
        \begin{aligned}
            |I_1| &\leq  C\eps^{1/2} h^{k+1} |\curl \ub|_{\Hb^{k+1}} \|\etab_h\|_{\mathrm{mad},h}, \\
            |I_4| &\leq C \delta^{1/2}\eps |\curl (\ub_I - \ub)|_{\Hb^1} \|\etab_h\|_{\mathrm{mad},h} \leq C \delta^{1/2}\eps h^{k} |\curl \ub|_{\Hb^{k+1}} \|\etab_h\|_{\mathrm{mad},h}.
        \end{aligned}
    \end{equation}
    Combining all the estimates and the triangle inequality, we can conclude.
\end{proof}

To balance the terms in~\eqref{eq:mad:err}, we adopt the standard choice of the stabilization parameter (see, e.g.,~\cite[][Eq.~3.38]{roos_2008}):
\begin{equation}\label{eq:mad:deltachoice}
    \delta = \begin{cases}
        \delta_0 h & \text{if } \mathrm{Pe} > 1 \text{ \emph{(advection-dominated regime)}},\\
        \delta_1 h^2/\eps & \text{if } \mathrm{Pe} \leq 1 \text{ \emph{(diffusion-dominated regime)}},
    \end{cases}
\end{equation}
    with $\delta_0, \delta_1 > 0$ being appropriate constants. Here, $\mathrm{Pe} := \|\betab\|_{L^\infty} h/(2\eps)$ is the P\'eclet number. With such choice, the error estimate~\eqref{eq:mad:err} becomes
\begin{corollary}[a priori error estimate]\label{thm:mad:errfinal}
    Under the same assumptions as in Theorem~\ref{thm:errsupg} and the choice of $\delta$ in~\eqref{eq:mad:deltachoice}, there exists a constant $C > 0$ independent of $h$ and $\eps$ such that
    \begin{equation}\label{eq:mad:errest}
        \|\ub - \ub_h\|_{\mathrm{mad},h} \leq 
             C h^{k} (h^{1/2} + \eps^{1/2}) |\ub|_{\Hb^{k+1}} + C \eps^{1/2} h^{k+1} |\curl \ub|_{\Hb^{k+1}}.
    \end{equation}
\end{corollary}

\section{Numerical experiments}\label{sec:num}
In this section, we present numerical experiments to verify the theoretical results and demonstrate the performance of the proposed schemes. Specifically, we first conduct a convergence study of~\eqref{eq:ccasymnonhomo2} for the $\curln$-elliptic problem~\eqref{eq:cc}. In particular, the different treatments of the boundary data discussed in Section~\ref{sec:cc:nonhomo} are compared. Besides, two failing cases are demonstrated: 1) using the \Nedelec elements of the second kind (see Remark~\ref{rmk:type2fail}); 2) using the \Nedelec elements of the first kind with a mesh violating the isolated patch condition (see Section~\ref{sec:patch}). Furthermore, we show the influence of the penalty parameter $C_p$ in the symmetric Nitsche's method~\eqref{eq:blsym} to stress the advantage of being penalty-free. In the second part, we conduct a convergence study of~\eqref{eq:supgscheme} for smooth solutions of the magnetic advection-diffusion problem~\eqref{eq:mad}. Then, different BC enforcement strategies (strong or weak) with/without stabilization are compared for solutions with boundary layers. 

All computations are implemented in \href{https://www.ngsolve.org}{\texttt{NGSolve}}.
\subsection{The curl-elliptic problem}
\subsubsection{Convergence study on smooth solution}\label{sec:num:cc:conv}
Let $\Omega = (0,1)^3$ and
\begin{equation}
    \ub^\ast(x,y,z) = [\sin(z + 1),\;\sin(x) + \cos(z),\;\sin(y)]^{\top}
\end{equation}
be the exact solution of the following equation:
\begin{equation}\label{eq:cc:numexp}
    \begin{aligned}
        \curl\curl \ub &= \fb &&\quad \text{in } \Omega,\\
        \div \ub &= 0 &&\quad \text{in } \Omega,\\
        \gammat \ub &= \gammat \ub^\ast &&\quad \text{on } \Gamma,
    \end{aligned}
\end{equation}
where $\fb = \curl\curl \ub^\ast$. We solve the problem using the scheme~\eqref{eq:ccasymnonhomo2} (labeled as \textbf{Interpolated BC}) and~\eqref{eq:ccasymnonhomo} (labeled as \textbf{Direct BC}). The convergence behaviors are demonstrated in Figure~\ref{fig:cc:convergence}. As discussed in Section~\ref{sec:cc:nonhomo}, the scheme~\eqref{eq:ccasymnonhomo2} with the interpolated BC data enjoys optimal convergence in $\Hcurl$ in agreement with Corollary~\ref{thm:cc:err}, whereas suboptimality occurs for the one~\eqref{eq:ccasymnonhomo} that plugs in the BC data directly.
\begin{figure}
    \centering
    \includegraphics[width=0.45\textwidth, clip, trim = 0 0 0 20]{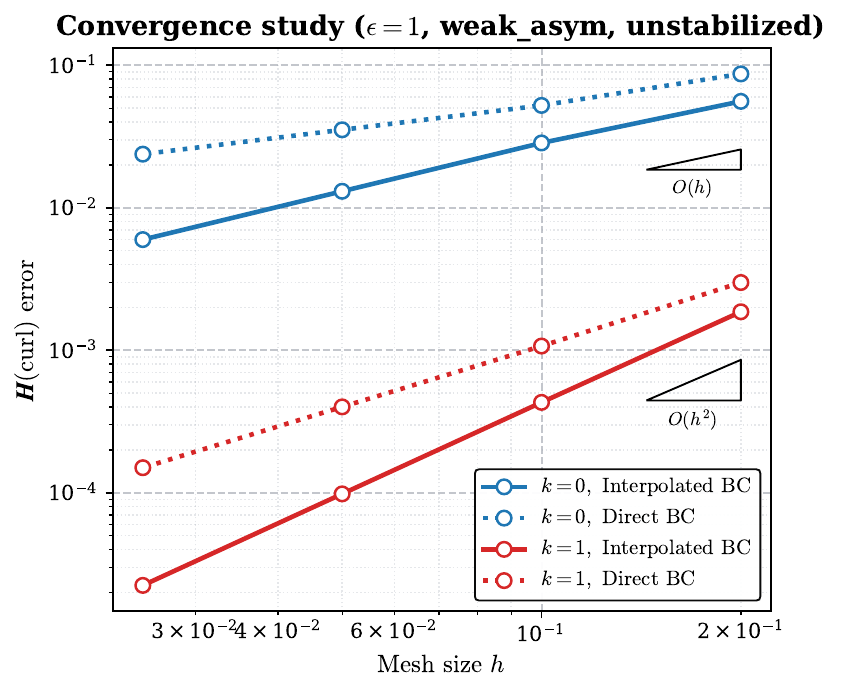}
    \caption{Convergence of the schemes for the curl-elliptic problem~\eqref{eq:cc:numexp} using different treatments of the boundary data: \textbf{Interpolated BC} for the scheme~\eqref{eq:ccasymnonhomo2} and \textbf{Direct BC} for the scheme~\eqref{eq:ccasymnonhomo}.}
    \label{fig:cc:convergence}
\end{figure}
\subsubsection{Failure of using the \Nedelec elements of the second kind} 
We solve the same problem~\eqref{eq:cc:numexp} but on a sphere $\Omega = \{\xb \in \R^3: \|\xb\| < 1\}$ to ensure that the underlying mesh meets the isolated patch condition strictly. In particular, we use the \Nedelec elements of the second kind (see Remark~\ref{rmk:type2fail}). Instability occurs as is shown in Figure~\ref{fig:cc:type2fail} (left). In contrast, the solution using the \Nedelec elements of the first kind is stable in Figure~\ref{fig:cc:type2fail} (right).
\begin{figure}
    \centering
    \includegraphics[width=0.3\textwidth]{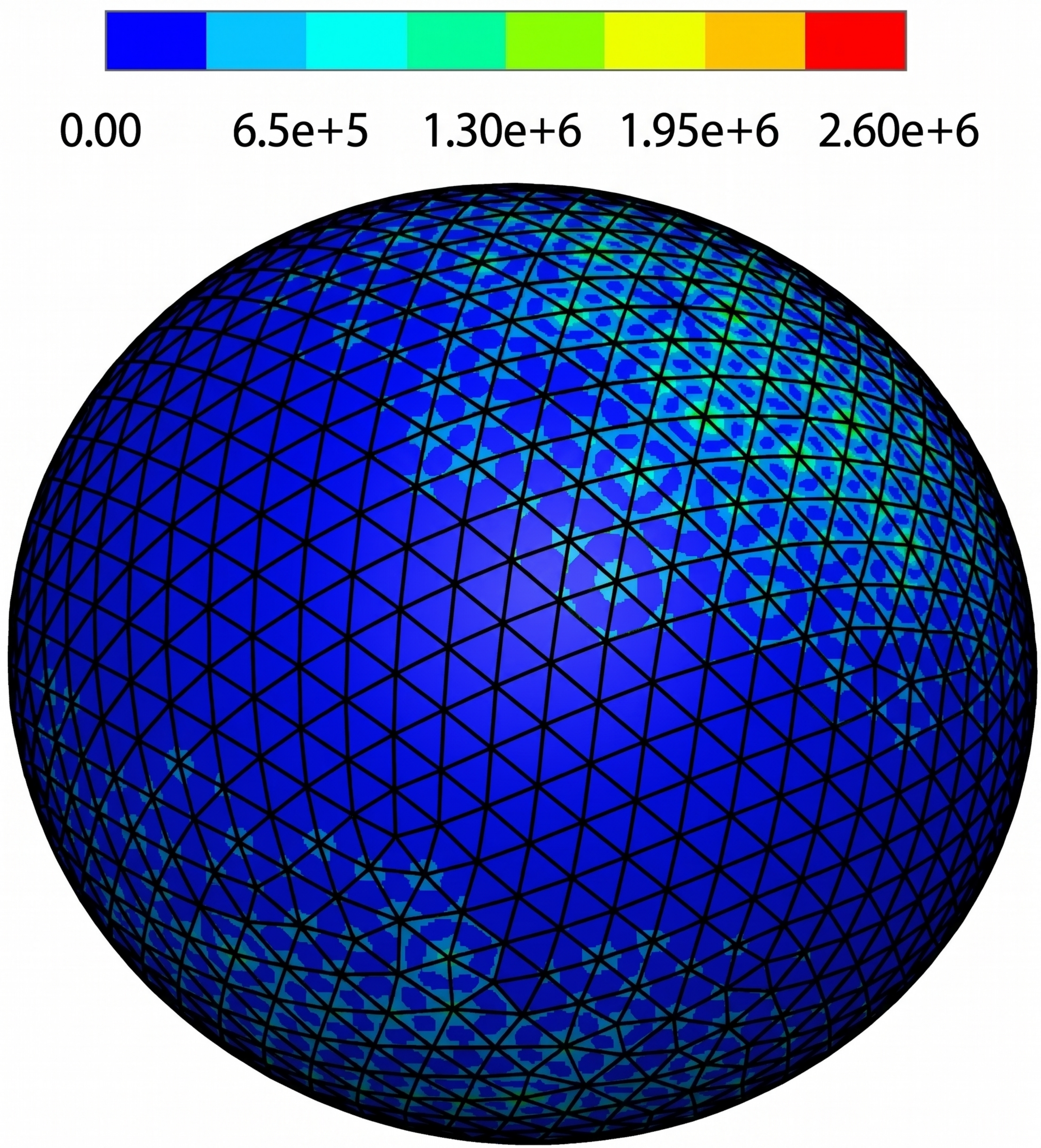}
    \hspace{2cm}
    \includegraphics[width=0.3\textwidth]{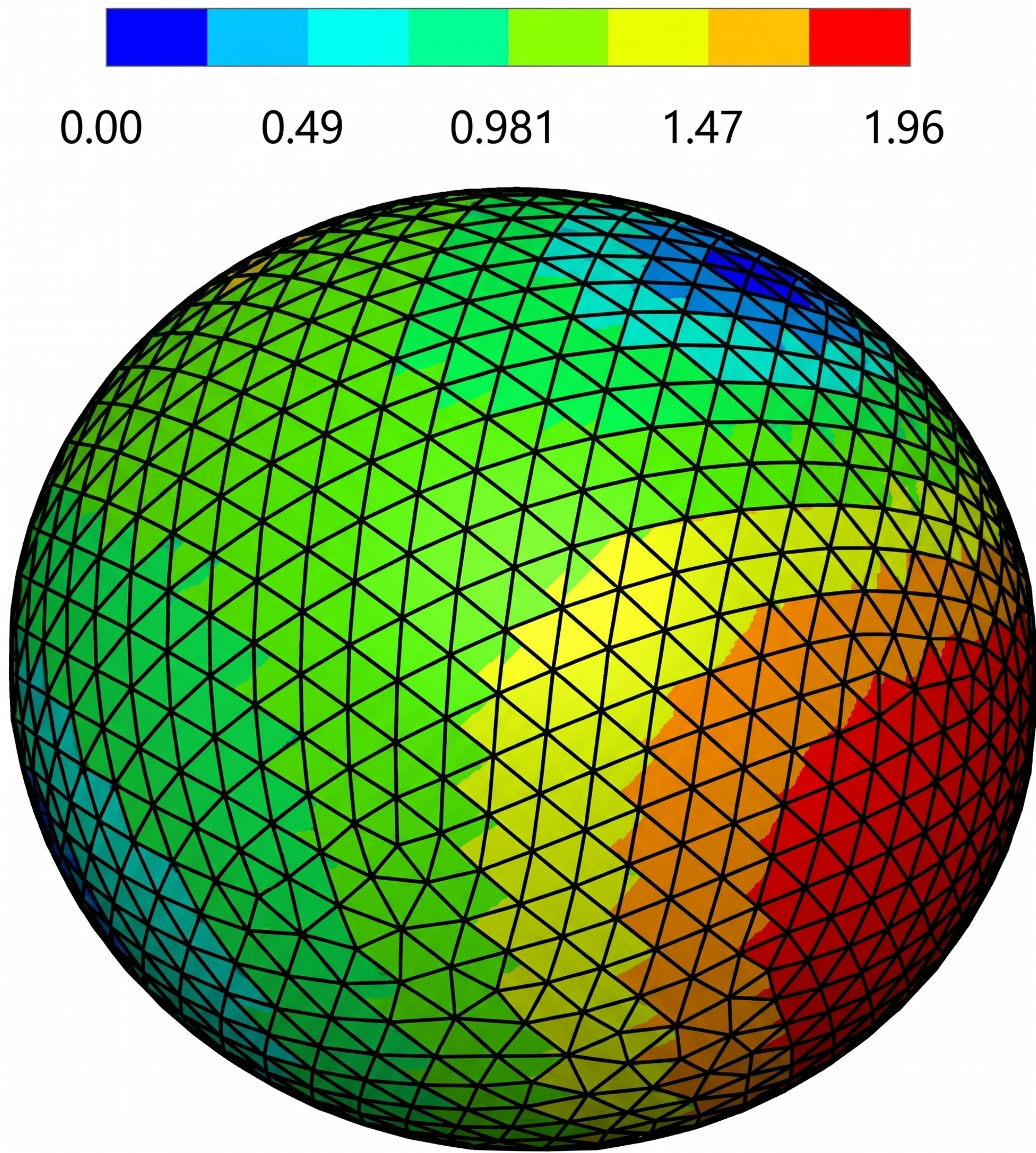}
    \caption{Failure of using the \Nedelec elements \emph{of the second kind} (left) and the stable solution using the \Nedelec elements \emph{of the first kind} (right). The mesh meets the isolated patch condition strictly. The color represents the magnitude of the solutions of~\eqref{eq:cc:numexp} with $\Omega = \{\xb \in \R^3: \|\xb\| < 1\}$. A mesh with $h = 0.1$ and an element order with $k=1$ (see~\eqref{def:spaces}) are used.}
    \label{fig:cc:type2fail}
\end{figure}

\subsubsection{Failure of using mesh violating the isolated patch condition}\label{sec:num:cc:meshfail}
We solve the same problem~\eqref{eq:cc:numexp} on a unit cube $\Omega = (0,1)^3$ using the \Nedelec elements \emph{of the first kind}. We test~\eqref{eq:ccasymnonhomo2} on a mesh containing tetrahedra that \emph{have three faces on $\Gamma$}. Instability is observed in Figure~\ref{fig:cc:ipcviolfail} (left). In contrast, the solution on a mesh that avoids such tetrahedra is stable in Figure~\ref{fig:cc:ipcviolfail} (right). 

We would like to point out that the mesh in Figure~\ref{fig:cc:ipcviolfail} (right) is generated via refining an initial mesh once (using \texttt{.Refine()} in \texttt{NGSolve} based on the bisection algorithm of~\cite{arnold_2000}). From our experience, instability induced by a bad mesh topology is effectively avoided by this approach, although no theoretical guarantee is promised. This indicates that a weaker condition than the isolated patch condition may be sufficient for stability.
\begin{figure}
    \centering
    \includegraphics[width=0.3\textwidth]{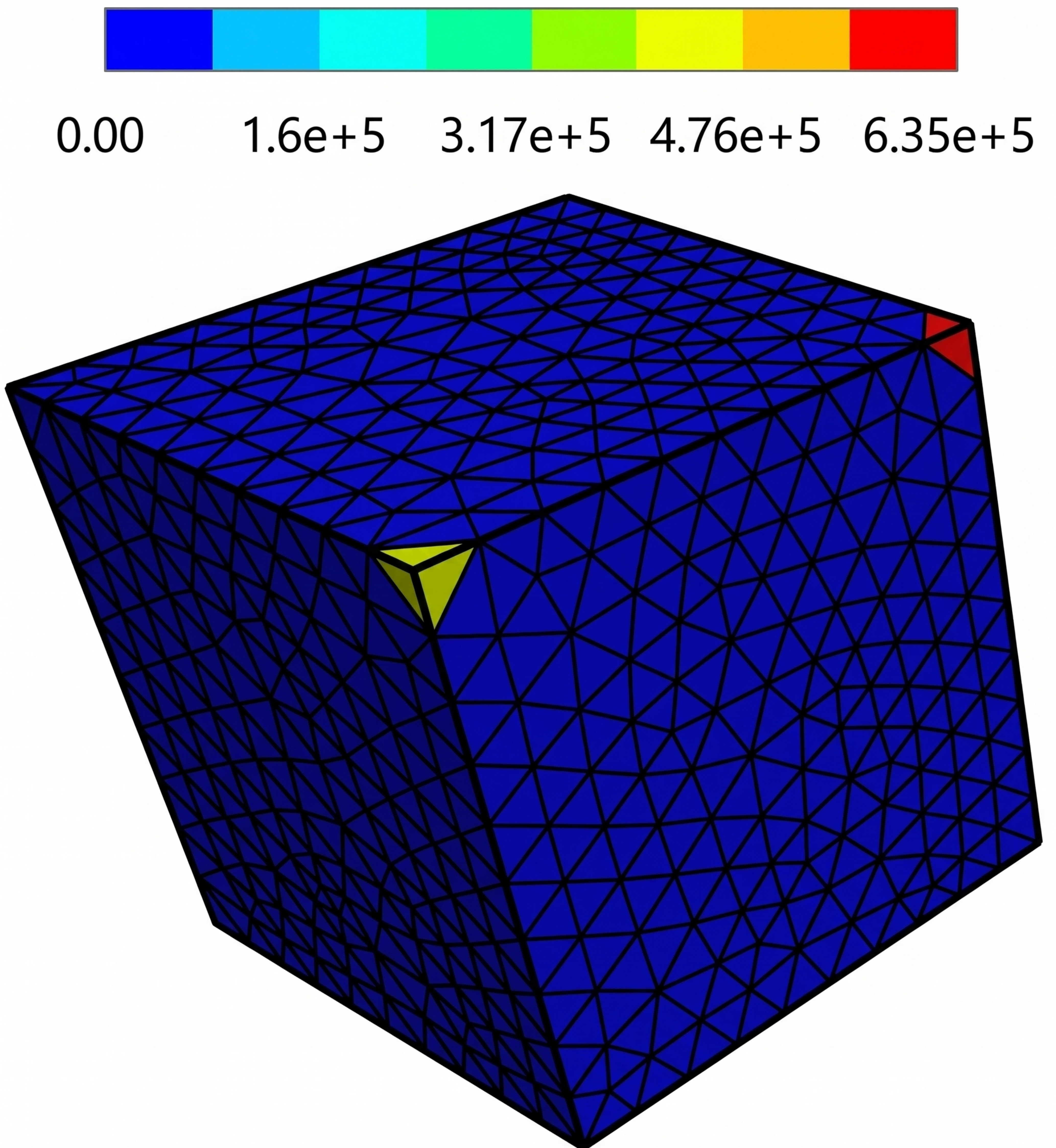}
    \hspace{2cm}
    \includegraphics[width=0.3\textwidth]{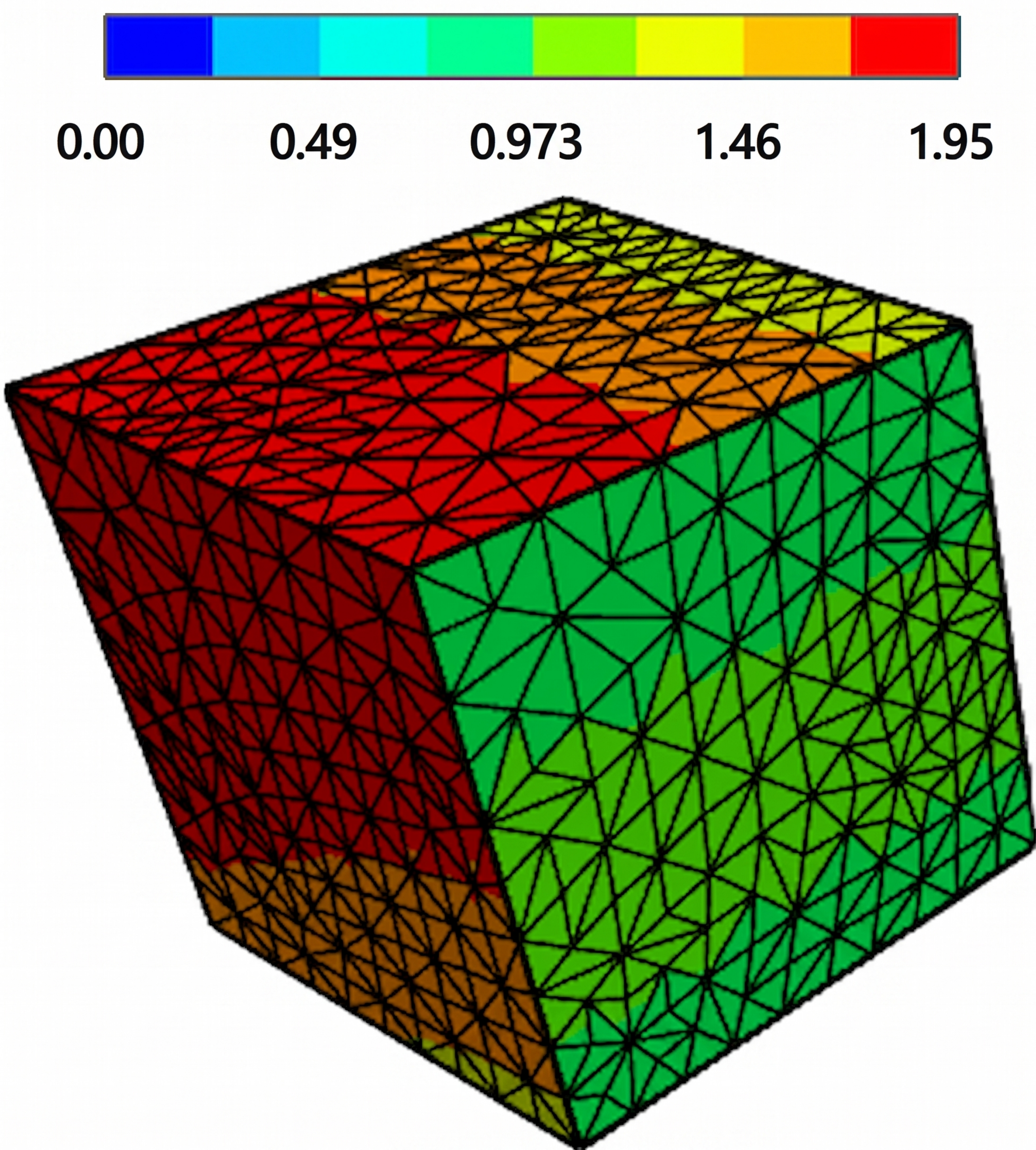}
    \caption{Failure of using a mesh \emph{that contains tetrahedra with three faces on $\Gamma$} (left) and the stable solution using a mesh \emph{whose tetrahedra have at most two faces on $\Gamma$} (right). The color represents the magnitude of the solutions of~\eqref{eq:cc:numexp} with $\Omega = (0,1)^3$. A mesh with $h = 0.1$ and an element order with $k=1$ (see~\eqref{def:spaces}) are used.}
    \label{fig:cc:ipcviolfail}
\end{figure}

\subsubsection{Influence of the penalty parameter}\label{sec:cc:penalty}
Recall the symmetric Nitsche's bilinear form $\a^{\mathrm{sym}}_{\curln,h}$ defined in~\eqref{eq:blsym}.
To solve~\eqref{eq:cc}, the symmetric Nitsche's method seeks $(\ub_h, p_h) \in \Sigmab_h \times \Pb_{h,0}$ such that
\begin{equation}\label{eq:cc:sym}
    \Bsf_{\curln}^\text{sym}((\ub_h, p_h), (\vb_h, q_h)) = (\fb, \vb_h) \quad \forall\,(\vb_h, q_h) \in \Sigmab_h \times \Pb_{h,0},
\end{equation}
where
\begin{equation}
    \begin{aligned}
        \B^{\mathrm{sym}}_{\curln}((\ub_h, p_h), (\vb_h, q_h)) := \a^{\mathrm{sym}}_{\curln, h}(\ub_h, \vb_h) + (\grad p_h, \vb_h) + (\grad q_h, \ub_h).
    \end{aligned}
\end{equation}
Compared with the asymmetric Nitsche's method~\eqref{eq:ccasym}, one needs to choose a penalty parameter $C_p$ in~\eqref{eq:cc:sym}. On one hand, a sufficiently large $C_p$ is required to guarantee stability by theory~\cite{casagrande_2016,boffi_2023,wouter_2026}. On the other hand, a large $C_p$ may lead to a large condition number of the resulting linear system. 

In Figure~\ref{fig:cc:penalty}, we plot errors in $\Hcurl$-norm of the scheme~\eqref{eq:cc:sym} using different values of $C_p$. The errors of the asymmetric Nitsche's method~\eqref{eq:ccasymnonhomo2} and the standard scheme~\eqref{eq:cc:std} are also plotted for comparison. Deterioration of the accuracy of the symmetric method is observed when $C_p$ is too small, which complies with the analysis. In contrast, the penalty-free asymmetric method is free of this issue and is as accurate as the symmetric one with an optimally chosen $C_p$. Similar phenomena are also observed in~\cite[][Figure~3]{wouter_2026} for the Stokes problem.
\begin{figure}
    \centering
    \includegraphics[width=0.45\textwidth, clip, trim = 0 0 0 20]{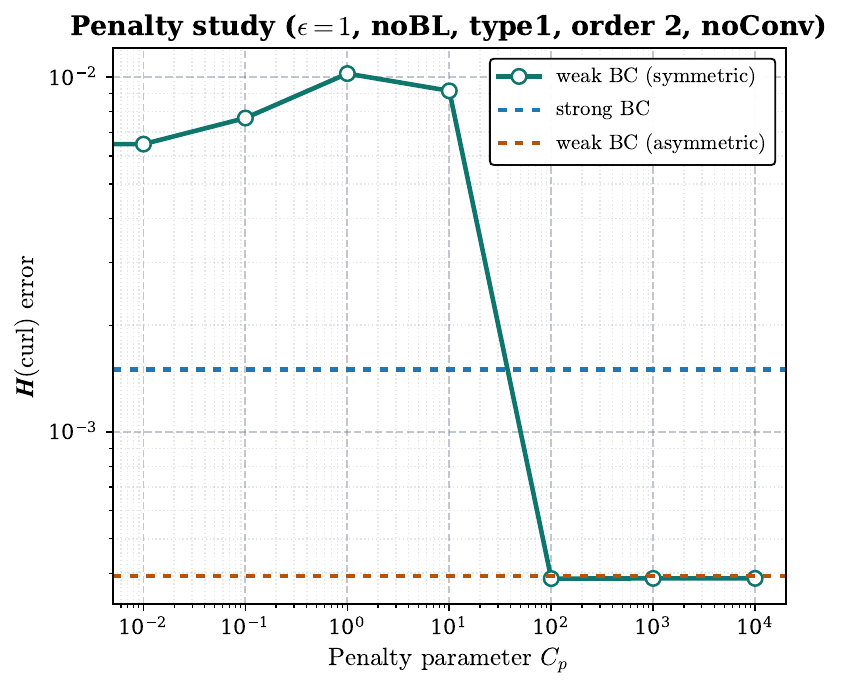}
    \caption{Influence of the penalty parameter $C_p$ in the symmetric Nitsche's method~\eqref{eq:cc:sym} for solving~\eqref{eq:cc:numexp} (labeled as \textbf{weak BC (symmetric)}). The errors of the asymmetric Nitsche's method~\eqref{eq:ccasymnonhomo2} (labeled as \textbf{weak BC (asymmetric)}) and the standard scheme~\eqref{eq:cc:std} (labeled as \textbf{strong BC}) are also plotted for comparison.  A mesh with $h = 0.1$ and an element order with $k=1$ (see~\eqref{def:spaces}) are used.}
    \label{fig:cc:penalty}
\end{figure}

\subsection{The magnetic advection-diffusion problem}\label{sec:num:mad}
\subsubsection{Convergence study on smooth solution}\label{sec:num:mad:conv}
Let $\Omega = (0,1)^3$ and
\begin{equation}
    \ub^\ast(x,y,z) = [y \exp(xz),\;- x^2y,\;\sin(xyz)]^{\top}
\end{equation}
be the exact solution of the following equation:
\begin{equation}\label{eq:mad:numexp}
    \begin{aligned}
        \eps \curl\curl \ub + \L_\betab \ub + \alpha \ub &= \fb &&\quad \text{in } \Omega,\\
        \gammat \ub &= \gammat \ub^\ast &&\quad \text{on } \Gamma\setminus\Gamma^-,\\
        \ub &= \ub^\ast &&\quad \text{on } \Gamma^-,
    \end{aligned}
\end{equation}
where $\eps = 10^{-6}$, $\betab(x,y,z)=[1 - z/2,\;2 + x,\;3 - y]^\top$, $\alpha=8$, and $\fb = \eps \curl\curl \ub^\ast + \L_\betab \ub^\ast + \alpha \ub^\ast$. We solve~\eqref{eq:mad:numexp} using the scheme~\eqref{eq:mad:nonhomo}. The convergence behaviors for $k=0$ and $k=1$ in $\Lb^2$-error are demonstrated in Figure~\ref{fig:mad:convergence}. According to the theoretical estimate~\eqref{eq:mad:errest}, the scheme is expected to converge with rate $O(h^{k+1/2})$ in $\Lb^2$-error for the advection-dominated regime ($\eps \ll h$). The observed convergence rate $O(h^{k+1})$ is approximately a half-order higher than predicted. This is a common phenomenon for advection-diffusion-type problems~\cite{burman_2006b,heumann_2013,li_2026} and a sharp convergence rate $O(h^{k+1/2})$ is typically observed only on special meshes~\cite{zhou_1997}. 

\begin{figure}
    \centering
    \includegraphics[width=0.45\textwidth, clip, trim = 0 0 0 20]{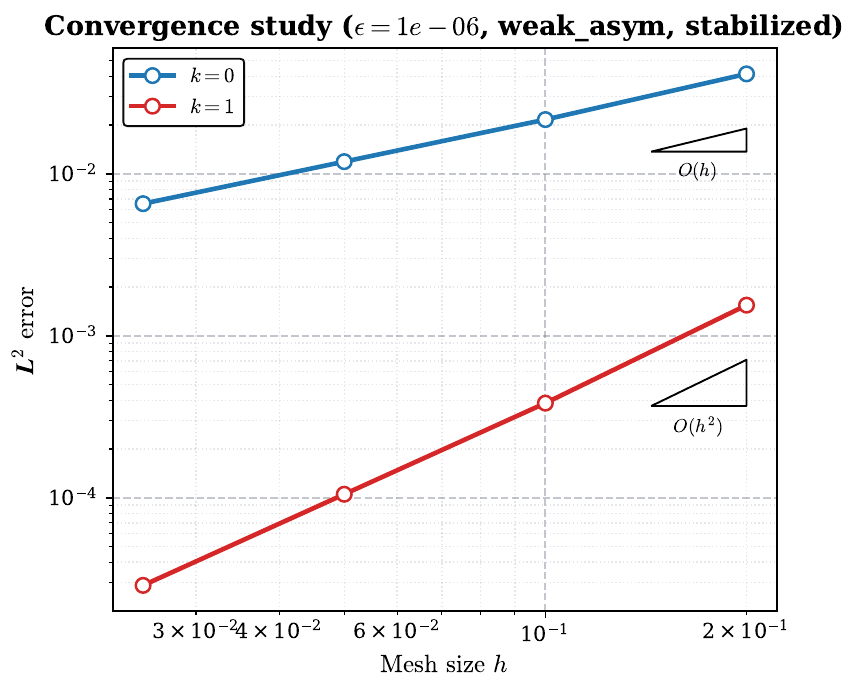}
    \caption{Convergence of the scheme~\eqref{eq:mad:nonhomo} for the magnetic advection-diffusion problem~\eqref{eq:mad:numexp}.}
    \label{fig:mad:convergence}
\end{figure}

\subsubsection{Solution with boundary layers}
Let $\Omega = (0,1)^3$. We consider the following equation:
\begin{equation}\label{eq:mad:numexpbl}
    \begin{aligned}
        \eps \curl\curl \ub + \L_\betab \ub + \alpha \ub &= \fb &&\quad \text{in } \Omega,\\
        \gammat \ub &= \zerob &&\quad \text{on } \Gamma\setminus \Gamma^-,\\
        \ub &= \ub^\ast &&\quad \text{on } \Gamma^-,
    \end{aligned}
\end{equation}
where $\eps, \alpha, \betab, \fb$ are the same as in~\eqref{eq:mad:numexp}. Due to the discrepancy $\gammat\ub \neq \gammat\ub^\ast$ on $\Gamma\setminus\Gamma^-$, the solution of~\eqref{eq:mad:numexpbl} is expected to display abrupt changes near $\Gamma\setminus \Gamma^-$, that is, boundary layers. In Figure~\ref{fig:mad:boundarylayer}, we compare the discrete solutions of~\eqref{eq:mad:numexpbl} using different BC enforcement strategies (\textbf{strong BC}~\eqref{eq:supgschemestrongbc} and \textbf{weak BC}~\eqref{eq:supgscheme}) \textbf{with} and \textbf{without} the stabilization terms $S^1$ and $S^2$ in~\eqref{eq:mag:stabterm}. Imposing the tangential BCs strongly without stabilization (Figure~\ref{fig:mad:strong-nostab}) leads to severe spurious oscillations. Adding the stabilization terms (Figure~\ref{fig:mad:strong-stab}) effectively suppresses the oscillations but is still subject to mild oscillations near the outflow boundary $\Gamma\setminus \Gamma^-$. In contrast, the solution with weakly enforced tangential BCs is free of spurious oscillations even without stabilization (Figure~\ref{fig:mad:weak-nostab}). Such a ``stabilization'' effect of weak imposition of BCs is well-known for scalar advection-diffusion problems~\cite{burman_2006b,schieweck_2008,franz_2011}. However, its error analysis predicts the reduced convergence rate $O(h^k)$ due to absence of stabilization. Therefore, weak imposition of tangential BCs combined with stabilization (Figure~\ref{fig:mad:weak-stab}) produces the most accurate results in the sense that it attains a convergence rate $O(h^{k+1/2})$ in the advection-dominated regime (see Corollary~\ref{thm:mad:errfinal}) and is free of spurious oscillations near the outflow boundary $\Gamma\setminus \Gamma^-$.
\begin{figure}
    \centering
    % a 2by2 figure
    \begin{subfigure}{0.45\textwidth}
        \includegraphics[width=\textwidth, clip, trim = 0 0 0 20]{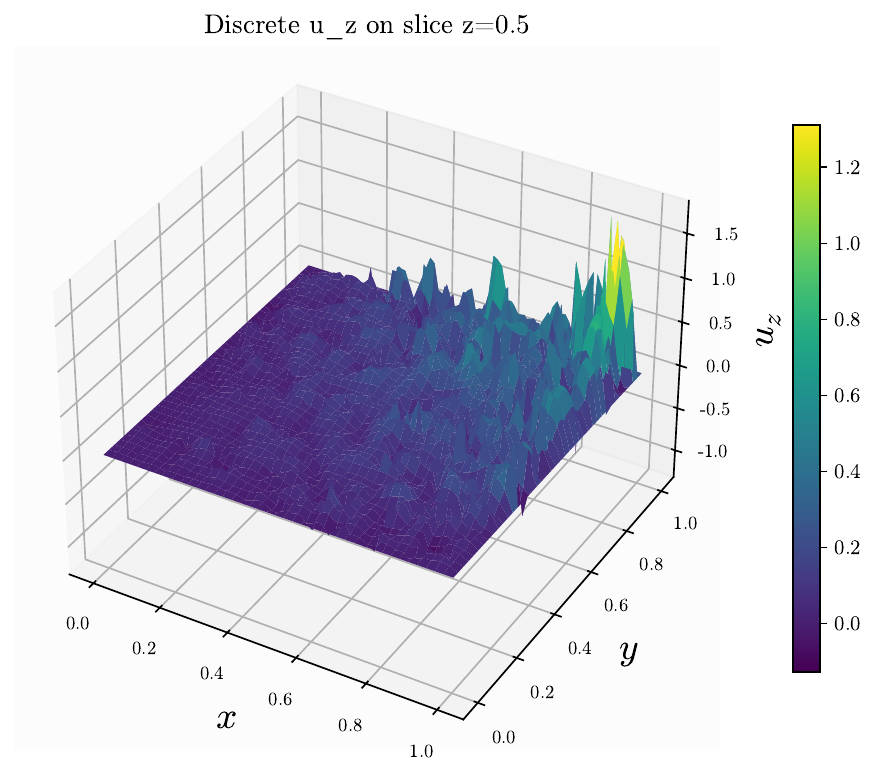}
        \caption{Strong BC with no stabilization.}
        \label{fig:mad:strong-nostab}
    \end{subfigure}
    \begin{subfigure}{0.45\textwidth}
        \includegraphics[width=\textwidth, clip, trim = 0 0 0 20]{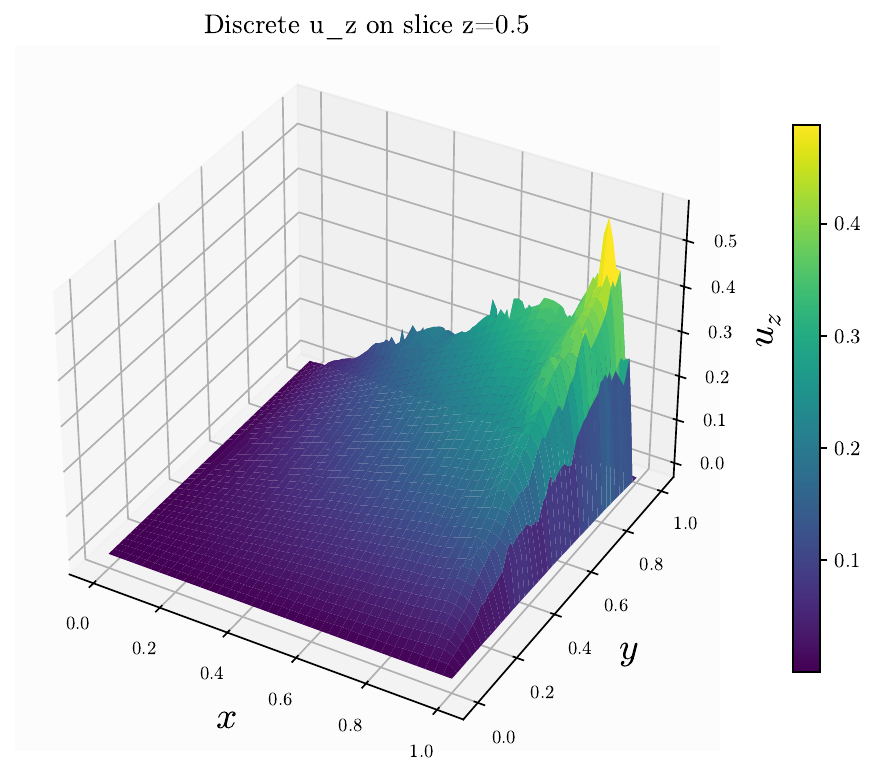}
        \caption{Strong BC with stabilization.}
        \label{fig:mad:strong-stab}
    \end{subfigure}
    \begin{subfigure}{0.45\textwidth}
        \includegraphics[width=\textwidth, clip, trim = 0 0 0 20]{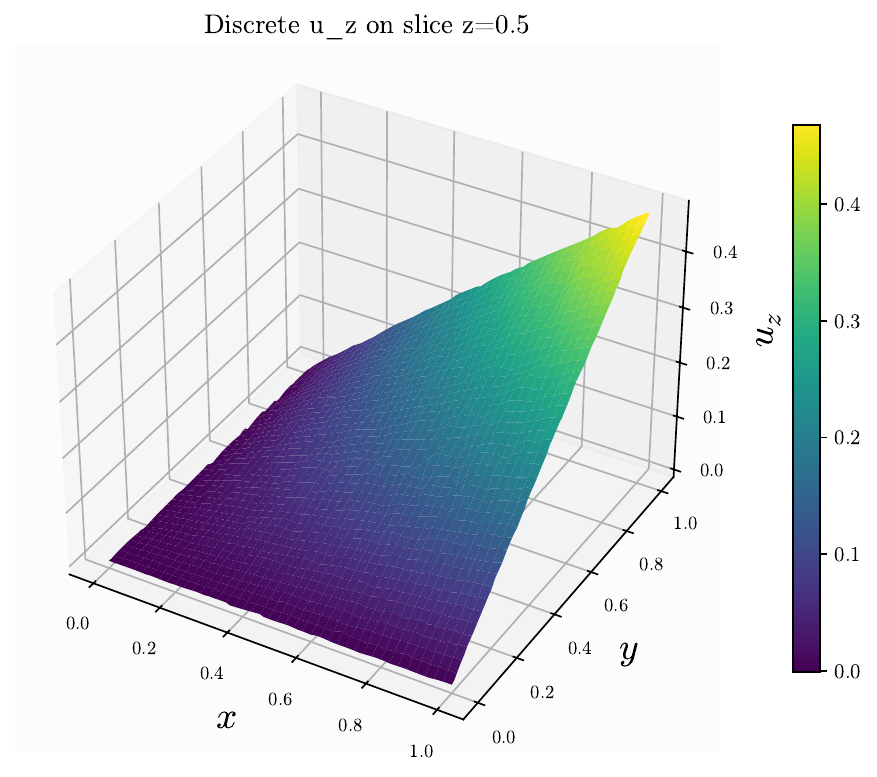}
        \caption{Weak BC with no stabilization.}
        \label{fig:mad:weak-nostab}
    \end{subfigure}
    \begin{subfigure}{0.45\textwidth}
        \includegraphics[width=\textwidth, clip, trim = 0 0 0 20]{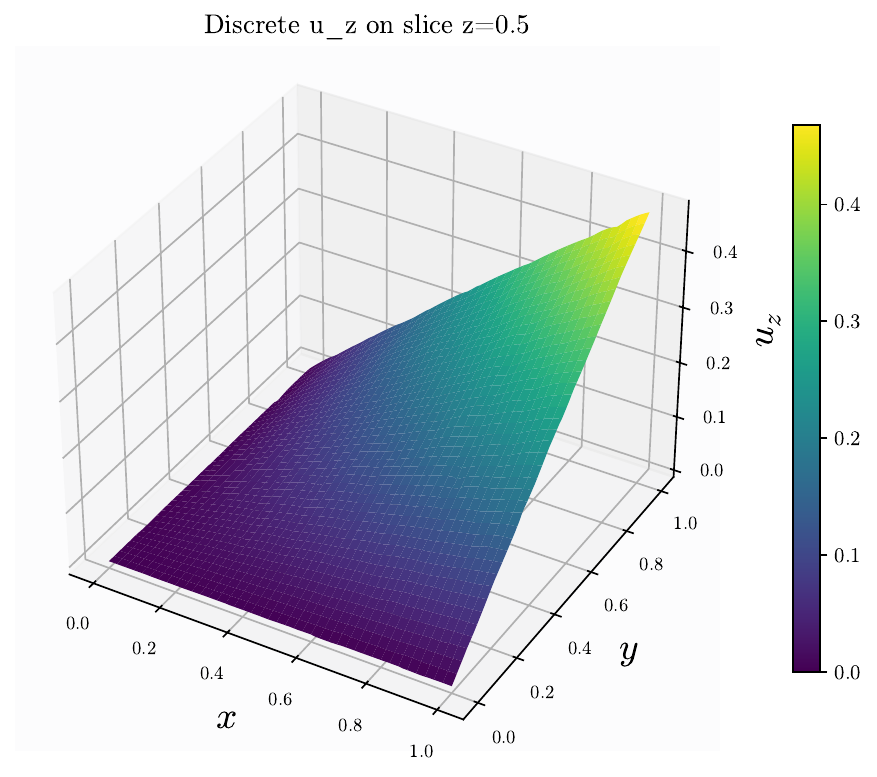}
        \caption{Weak BC with stabilization.}
        \label{fig:mad:weak-stab}
    \end{subfigure}
    \caption{Comparison of the discrete solutions of~\eqref{eq:mad:numexpbl} with tangential BCs strongly~\eqref{eq:supgschemestrongbc} or weakly~\eqref{eq:supgscheme} enforced and with or without the stabilization terms $S^1$ and $S^2$ in~\eqref{eq:mag:stabterm}. The color and height represent the $z$-component of the solutions on a two-dimensional slice at $z = 0.5$. A mesh with $h = 0.1$ and an element order with $k=1$ (see~\eqref{def:spaces}) are used.}
    \label{fig:mad:boundarylayer}
\end{figure}

\clearpage
\section{Concluding remarks}\label{sec:conclusion}
We have shown the inf-sup stability of the penalty-free asymmetric Nitsche's bilinear form $\a^{\mathrm{asym}}_{\curln,h}$ in~\eqref{eq:blasym} (see Lemma~\ref{thm:infsupasym}). This provides a valid approach for discretizing the $\curln$-$\curln$ operator with the tangential BCs imposed weakly. Its symmetric variant based on $\a^{\mathrm{sym}}_{\curln,h}$ in~\eqref{eq:blsym} requires choosing a penalty parameter $C_p$. A poor choice of $C_p$ can deteriorate the accuracy (see Section~\ref{sec:cc:penalty}). In contrast, the penalty-free asymmetric Nitsche's method avoids this parameter selection. Its accuracy for the two model problems~\eqref{eq:cc} and~\eqref{eq:mad} is confirmed by both the analysis (Section~\ref{sec:cc:err} and~\ref{sec:mad:err}) and the experiments (Section~\ref{sec:num:cc:conv} and~\ref{sec:num:mad:conv}). 

Special care should be taken with the restrictions on the edge element spaces and mesh topology. The analysis relies on \Nedelec elements of \emph{the first kind} (see definition in~\eqref{def:spaces}) and on meshes satisfying the \emph{isolated patch condition} (see Section~\ref{sec:patch}). Instability of using \Nedelec elements of the second kind is confirmed numerically (see Figure~\ref{fig:cc:type2fail}). Failure on meshes violating the isolated patch condition is also observed (see Figure~\ref{fig:cc:ipcviolfail}), while this failure can be effectively avoided by a reasonable mesh generation strategy (see Section~\ref{sec:num:cc:meshfail}).

Overall, the penalty-free asymmetric Nitsche's method is a viable alternative for discretizing $\curln$-$\curln$-type operators. Following its $H^1$-counterpart for Laplacian-type operators, its application extends to many other contexts, such as interface problems and unfitted mesh methods, with potential advantages over its symmetric variant. 

Fianlly, we remark that the $\Hdiv$-counterpart of the penalty-free asymmetric Nitsche's method for $\gradn$-$\divn$-type operators can be analyzed in a rather similar manner. 
 
\bibliography{ref.bib}
\end{document}